\documentclass[11pt,reqno]{amsart}
\usepackage[utf8]{inputenc}
\usepackage{latexsym,amssymb,longtable}
\usepackage{hyperref}

\usepackage{amsmath,amsthm,amsxtra}
\usepackage{amsfonts,eucal,mathrsfs}
\usepackage{enumitem}
\usepackage{a4wide}
\usepackage{xcolor}
\usepackage{upgreek}

\numberwithin{equation}{section}

\newcommand{\N}{\mathbb{N}}

\newcommand{\R}{\mathbb{R}}

\newcommand{\eps}{\varepsilon}
\newcommand{\opt}{\mathrm{opt}}

\newcommand{\dd}{\mathrm{d}}

\newtheorem{theorem}{Theorem}[section]

\newtheorem{corollary}[theorem]{Corollary}
\newtheorem{lemma}[theorem]{Lemma}
\newtheorem{proposition}[theorem]{Proposition}

\theoremstyle{definition}
\newtheorem{definition}[theorem]{Definition}
\newtheorem{remark}[theorem]{Remark}

\newtheorem{example}[theorem]{Example}

\let\eps\ep

\newcommand{\foraa}{\text{for a.e.\ }}

\def\dd{\mathrm{d}}
\DeclareMathSymbol{\mtimes}{\mathord}{symbols}{"0A}
\newcommand{\argmin}{\mathop{\rm argmin}}

\newcommand{\DVT}[3]{\DVTn(#1,#2,#3)}

\newcommand{\DVTn}{\Wcostname}
\newcommand{\VarW}[3]{\mathbb A(#1;[#2,#3])}
\newcommand{\VarWd}[3]{\mathbb A_2(#1;[#2,#3])}
\newcommand{\VarWdname}{\mathbb A_2}

\newcommand{\Var}{\mathrm{Var}}
\newcommand{\VarWname}[1]{\mathbb A (#1)}

\newcommand{\BV}{\mathrm{BV}}
\newcommand{\AC}{\mathrm{AC}}

\newcommand{\dvt}{\mathrm{DVT}}

\newcommand{\weakto}{\rightharpoonup}

\newcommand{\piecewiseConstant}[2]{\overline{#1}_{\kern-1pt#2}}

\newcommand{\underpiecewiseConstant}[2]{\underline{#1}_{\kern-1pt#2}}

\newcommand{\Utau}{\xcurve_\tau}
\newcommand{\pwM}[2]{\widetilde{#1}_{\kern-1pt#2}}
\def\gen(#1,#2){\calS_{#1}(#2)}

\newcommand{\teta}{\boldsymbol \vartheta}

\newcommand{\pinfty}{{+\infty}}
\newcommand{\mres}{\kern1pt\mathbin{\vrule height 1.6ex depth 0pt width
    0.13ex\vrule height 0.13ex depth 0pt width 1.3ex}}

\def\calS{\mathscr E}

\newcommand{\restr}[1]{\lower3pt\hbox{$|_{#1}$}}
\newcommand{\nchi}{{\raise.3ex\hbox{$\chi$}}}

\newcommand{\jj}{{\boldsymbol{j}}}

\newcommand{\Xsp}{\boldsymbol{X}}
\newcommand{\xv}{u}
\newcommand{\xcurve}{U}
\newcommand{\Wcostname}{\mathsf{a}}
\newcommand{\Wcostsupname}{\Wcostname_{\rm sup}}
\newcommand{\Wcost}[3]{\Wcostname(#1,#2,#3)}
\newcommand{\Wcostplus}[3]{\Wcostname_+(#1,#2,#3)}
\newcommand{\Wcostminus}[3]{\Wcostname_-(#1,#2,#3)}
\newcommand{\Wcostsup}[2]{\Wcostsupname(#1,#2)}

\newcommand{\tosi}{\stackrel{\sigma}{\rightarrow}}

\newcommand{\down}{\downarrow}
\newcommand{\up}{\uparrow}

\newcommand{\smd}[2]{\mathfrak{a}[#1'](#2)}
\newcommand{\smdn}[1]{\mathfrak{a}[#1']}

\newcommand{\mder}[3]{|{#2}'|_{\mathsf{d}_{#1}}(#3)}

\newcommand{\Wmetric}{\mathsf d_{\Wcostname}}
\newcommand{\Wmetricl}[1]{\mathsf d_{\Wcostname,#1}}

\newcommand{\metr}[3]{d_{#1}(#2,#3)}
\newcommand{\metrname}[1]{d_{#1}}
\newcommand{\Psidens}{\uppsi}
\newcommand{\psidens}[1]{\uppsi_{#1}}

\newcommand{\Gacostname}{\mathsf {b}}
\newcommand{\Gacost}[3]{\Gacostname(#1,#2,#3)}
\newcommand{\tpart}[1]{\overline{\mathsf{t}}_{#1}}
\newcommand{\utpart}[1]{\underline{\mathsf{t}}_{#1}}
\newcommand{\fpart}[2]{\overline{\mathfrak{#1}}_{#2}}
\newcommand{\ufpart}[2]{\underline{\mathfrak{#1}}_{#2}}

\newcommand{\Nei}[3]{\mathrm{U}(#1;#2,#3)}
\newcommand{\VNei}[2]{\mathrm{V}(#1,#2)}

\newcommand{\TuNei}[2]{\mathrm{V}(#1,#2)}

\newcommand{\Wtop}{\mathfrak{O}}
\newcommand{\toW}{\stackrel{\Wcostname}{\rightarrow}}
\newcommand{\llim}[2]{{#1}_-(#2)}
\newcommand{\rlim}[2]{{#1}_+(#2)}
\newcommand{\Tvar}[2]{\mathsf{#1}_{#2}}
\newcommand{\finslname}{\mathcal{R}}
\newcommand{\finsl}[2]{\finslname(#1,#2)}
\newcommand{\actname}{\mathscr{R}}
\newcommand{\act}[2]{\actname(#1,#2)}
\newcommand{\diad}[1]{\mathcal{P}_{#1}}
\newcommand{\did}{\mathrm{diad}}

 \newcommand{\nc}{\normalcolor}
 
 \definecolor{pink}{rgb}{0.7,0,0.6}
\definecolor{lmagenta}{rgb}{0.8,0.4,0.6}
\definecolor{ddcyan}{rgb}{0.4,0.5,1}
\definecolor{violet}{rgb}{0.4,0,0.9}

\usepackage[textsize=footnotesize]{todonotes}

 \def\Xint#1{\mathchoice
{\XXint\displaystyle\textstyle{#1}}%
{\XXint\textstyle\scriptstyle{#1}}%
{\XXint\scriptstyle\scriptscriptstyle{#1}}%
{\XXint\scriptscriptstyle\scriptscriptstyle{#1}}%
\!\int}
\def\XXint#1#2#3{{\setbox0=\hbox{$#1{#2#3}{\int}$ }
\vcenter{\hbox{$#2#3$ }}\kern-.6\wd0}}

\def\dashint{\Xint-}

\numberwithin{equation}{section}

\title[]{Abstract Action Spaces and their topological and dynamic properties}

\author{Riccarda Rossi}
\address{Riccarda Rossi, DIMI, Universit\`a degli studi di Brescia. Via Branze 38, I--25133 Brescia -- Italy}
\email{riccarda.rossi\,@\,unibs.it}

\author{Giuseppe Savar\'e}
\address{Giuseppe Savar\'e, Department of Decision Sciences and BIDSA,
  Bocconi University.
  Via Roentgen 1, I--20136 Milan -- Italy}
\email{giuseppe.savare\,@\,unibocconi.it}

\thanks{R.R.\ acknowledges support from the  PRIN project
\emph{PRIN 2020: ``Mathematics for Industry 4.0"}. 
 G.S.\ acknowledges support from the  PRIN project \emph{``PRIN 202244A7YL: Gradient Flows and Non-Smooth Geometric Structures with Applications to Optimization and Machine Learning”.}}

\begin{document}

\maketitle

\begin{abstract}
 We introduce the concept of \emph{action space}, a set $\Xsp$ endowed with an action cost
$\mathsf a:(0,+\infty)\times \Xsp\times \Xsp\to [0,+\infty)$
satisfying suitable axioms, which turn out to  provide  a `dynamic' generalization of 
the classical notion of metric space.
Action costs naturally arise as dissipation terms featuring in the Minimizing Movement scheme for gradient flows, which can then be settled in general action spaces. 
\par
As in the case of metric spaces,
we will show that action costs induce an intrinsic topological and metric structure on $\Xsp$. Moreover, we introduce the related action functional on paths in $\Xsp$, investigate the properties  of curves of finite action, and discuss their absolute continuity. Finally,  under a condition akin to the \emph{approximate mid-point property} for metric spaces, we provide a dynamic interpretation of action costs. 
\end{abstract}

\begin{center}
{\sl Dedicated to Pierluigi Colli on the occasion of his 65th birthday}
\end{center}

%

\medskip


\section{Introduction}
\noindent
Since the pioneering work by \textsc{E.\ De Giorgi} \cite{DG93}, which was inspired the approach by \textsc{Almgren}, \textsc{Taylor} and \textsc{Wang}
 to the mean curvature and other geometric
flows,
Minimizing Movements have become a paradigmatic tool
for constructing solutions to  a large class of evolutionary problems. 

In  its  full generality, 
the Minimizing Movement scheme consists in finding, for a given
time interval $[0,T]$ and a time step   $\tau>0$
inducing the uniform grid
$0<\tau<2\tau<\cdots<N_\tau \tau$,
with $N_\tau\in \N$ such that $(N_\tau-1)\tau<T\le N_\tau \tau$,
discrete solutions  $(\Utau^n)_{n=0}^{N_\tau}$
in some topological space $\Xsp$, 
as solutions of the following recursive family of 
minimum problems
\begin{equation}
\label{MM_general} \Utau^n \in \argmin_{V \in \Xsp} 
\mathscr F(\tau, n, U^{n-1}_\tau;V), \qquad
\text{$n=1, \ldots, {N_\tau},$}
\end{equation}
with $\Utau^0=u_0$ a given initial
datum,
where $\mathscr F:(0,+\infty)\times\N\times \Xsp\times \Xsp\to \R\cup\{+\infty\}$
is a suitable functional.
Minimizing Movements are limits
as $\tau\down0$ of the piecewise
constant interpolations $U_\tau$
of the values $U^n_\tau$,  viz.\ 
$U_\tau: [0,T]\to \Xsp$, 
$U_\tau(t):=U_\tau^n$
if $t\in ((n-1)\tau,n\tau].$

A   particularly significant  example
arises in the case
of gradient flows
of a  time-dependent functional
$\mathscr{E} :[0,T]\times \Xsp\to\R\cup\{+\infty\}$, 
with respect to a metric $d$
on $\Xsp$:  the related Minimizing Movement scheme
 corresponds
to a functional $\mathscr F$ of the form
%
%
%
%
\begin{equation}
\label{MM_GFLOW} 
\mathscr F(\tau,n,U;V):=
\frac1{2\tau}d^2( U,V) +\mathscr E(n\tau,V).
\end{equation}
\nc 
Under suitable assumptions on $\calS$, the approximate solutions defined by interpolation
of the discrete values $(\Utau^n)_{n=0}^{N_\tau}$  converge to 
a \emph{curve of maximal slope}  \cite{Ambr95MM, AGS08, RMS08}. 
\par
In the applications, gradient flows model processes whose temporal evolution results from the trade-off
of energy conservation and energy dissipation. Dissipative mechanisms are then encoded in the metric $d$. In
the general case, dissipation may be mathematically modelled by functionals of the type $\uppsi(d)$, with 
$\Psidens :[0,+\infty)\to [0,+\infty)$ a convex function null at $0$ 
(the gradient-flow case corresponding to $\Psidens(r) = \tfrac12 r^2$). 
The corresponding
functional generating the
Minimizing Movement scheme, i.e.
%
%
\begin{equation}
\label{MM_DNE} 
\mathscr F(\tau,n,U;V):=
\tau \uppsi  \left( \frac{d(U,V)}\tau \right)  +\mathscr E(n\tau,V)\,,
\end{equation}
 for a general  (convex)  $\uppsi$ with \emph{superlinear growth}  at infinity,
 has been tackled in  \cite{RMS08}
 (cf.\ also \cite{Zimmer-metric}).
The linear-growth case $\Psidens(r)=r$ falls into the realm of rate-independent evolution
\cite{MaiMie05EREM,MieRouBOOK}.
\par
It is then natural
(see \cite{PRST22}) to study 
more general 
MM-functionals $\mathscr F$
of the form
\begin{equation}
\label{eq:time-incremental-intro} \mathscr F(\tau,n,U;V):=
\Wcostname(\tau, U,V) +\mathscr E(n\tau , V),
\end{equation}
where
$\Wcostname: (0,+\infty)\times \Xsp\times \Xsp \to [0,+\infty)$ 
can be interpreted as a sort of action functional, measuring the 
cost  for moving from  the point $U$ to the point $V$ 
in the amount of time $\tau>0$.  The structure
\eqref{eq:time-incremental-intro}
still preserves the natural splitting between 
a  driving   energy functional $\mathscr E$
and a metric-like   dissipation 
functional $\Wcostname$,
which however is not derived from a   given  metric $d$ on $\Xsp$.

The structural property 
which takes into account the heuristic interpretation of $\Wcostname$
is the 
\textbf{concatenation inequality}
\begin{equation}\label{e:psi3-intro}
 \Wcostname(\tau_1+\tau_2,\xv_1,\xv_3) \leq  \Wcostname(\tau_1,\xv_1,\xv_2) +  \Wcostname(\tau_2,\xv_2,\xv_3)
 \quad u_i\in \Xsp,\ \tau_i>0,
\end{equation}
which may also be interpreted as a
dynamic version of the triangle inequality  for  a metric.
Still inspired by the axioms of 
metrics, we will focus on 
actions that 
\textbf{vanish only on the diagonal of $\Xsp\times \Xsp$}
and are \textbf{symmetric}, i.e.
%
for all $\tau>0,\, \xv_0,\, \xv_1 \in \Xsp$,	\begin{equation}\label{e:psi2-intro}
\Wcostname(\tau,\xv_0,\xv_1)\ge0;\qquad	\Wcostname(\tau,\xv_0,\xv_1)= 0 \ \Leftrightarrow \ \xv_0=\xv_1,
  \end{equation}
 \begin{equation}
\label{e:psi1-intro}
\Wcost{\tau}{\xv_1}{\xv_2} = \Wcost{\tau}{\xv_2}{\xv_1}.
\end{equation}
 	%
%
Hereafter, we will term  \emph{action cost}  any function 
complying with \eqref{e:psi3-intro},  \eqref{e:psi2-intro}, \&
\eqref{e:psi1-intro}
and \emph{action space} 
a pair $(\Xsp,\Wcostname)$ given by a set
$\Xsp$ and an   action cost  $\Wcostname$.
Under further compatibility conditions between $\Wcostname$ and the driving energy functional 
$\calS$, in \cite{PRST22} it has been shown that the approximate solutions arising from the Minimizing Movement scheme generated by
 the functional 
$\mathscr{F}$ from  \eqref{eq:time-incremental-intro} converge to a curve fulfilling a suitable Energy-Dissipation (in)equality that in fact generalizes the metric formulation of gradient flows.
In particular, the Minimizing Movement scheme \eqref{eq:time-incremental-intro} can be set up in the general framework
of action spaces, in which 
$\Wcostname$ is not induced by an underlying metric on the ambient space $\Xsp$. 


It is easy to check that properties 
\eqref{e:psi3-intro},  \eqref{e:psi2-intro}, \&
\eqref{e:psi1-intro} 
are satisfied by all functions of the form
 $\Wcostname(\tau, u,v) =\tau \Psidens \left( \tfrac1\tau d(u,v)\right)$,
 where $\Psidens:[0,+\infty)\to[0,+\infty)$
 is a convex function vanishing only at $0.$
 The class of   action costs,  however, 
 is much larger and includes 
 diverse  functionals.
 Paradigmatic examples
 are provided by 
  costs   arising from  the minimization of a suitable action functional:   when $\Xsp=\R^d$ we can define
 \begin{equation}
 \label{Finsler-cost-intro}
 \Wcostname(\tau,u,v): = \inf \left\{ \int_0^{\tau} \finsl{\Theta(r)}{\Theta'(r)} \dd r \, : \ \Theta \in \AC([0,\tau];\R^d)\, \ \Theta(0)=u, \ \Theta(\tau)= v  \right\},
 \end{equation} 
 with $\finslname: \R^d\times \R^d \to [0,+\infty)$ such that $\finsl{\Theta}{\cdot}$ is convex, superlinear 
 and vanishes only at $0$
 (cf.\ Sec.\  \ref{s:example} for all details).
 We emphasize that, while Riemannian-Finsler metrics  are defined by infimizing an action integral which is positively $1$-homogeneous with respect to the velocity  variable, 
   here we allow for general convex integrands. 
 \par
 The other main motivating example is provided by the so-called \emph{Dynamical-Variational Transport} ($\dvt$) costs, which generalize transport distances between measures. 
 Their definition in \cite{PRST22}
has been indeed inspired by the well-known dynamic reformulation of the Wasserstein distance $W_2$ advanced by \textsc{Benamou} \& \textsc{Brenier} \cite{Benamou-Brenier}, cf.\ also \cite{DolbeaultNazaretSavare09,Maas11,Mielke13CALCVAR}.
 Another notable  and inspiring construction for the present paper has been proposed in
\cite{Figalli-Gangbo-Yolcu11}, 
where the authors extended the Minimizing Movement scheme, and \textsc{De Giorgi}'s interpolation techniques
\cite{DG93,Ambr95MM},
 to carry out a variational analysis of PDEs that are not gradient flows but, still, possess an entropy functional and an underlying Lagrangian.
 It is in terms of this Lagrangian, which also depends on the spatial variable, that they defined an action integral and, ultimately, an action cost 
 for the Minimizing Movement scheme.
\par
   Now, in the same way as the  `action integral cost' 
    $\Wcostname$ \eqref{Finsler-cost-intro} generalizes the standard construction of  Riemann-Finsler metrics, 
    so do $\dvt$ costs extend transport distances  between measures. In fact, given two positive finite measures $\mu_0,\mu_1$ on $\R^d$ (for simplicity; more general state spaces 
    have been considered in \cite{PRST22}), 
   the cost for connecting them over a certain interval $[0,\tau]$ is defined by 
  minimizing a suitable action integral over curves of measures joining $\mu_0$ to $\mu_1$ and solving the continuity equation on $(0,\tau)$ with flux $\jj$:
   \begin{equation}
   \label{dvt-intro}
   \Wcost \tau{\mu_0}{\mu_1} : = \inf\left \{ \int_0^\tau \act{\rho(r)}{\jj(r)} \dd r\, : \ (\rho,\jj) \in \mathcal{CE}(0,\tau), \  \rho(0)=\mu_0, \ \rho(1) = \mu_1\right\}
   \end{equation}
    (where $ \mathcal{CE}(0,\tau)$  denotes the family of solutions to the continuity equation on $(0,\tau)$). 
   \par
    Because of the flexibility and frequent occurrence of  action structures in the variational approach to evolutionary problems, 
     we believe that action spaces
     deserve to be studied in their own right: 
     like in the case of metric spaces,
     we will show that 
     they induce a natural (metrizable) topology,
     an intrinsic notion of completeness, and  canonical action functionals
     on $\Xsp$-valued paths.
  \subsection*{Plan of the paper}  
In this note we develop a systematic analysis
of a space $\Xsp$  endowed with   an action cost  $\Wcostname$. 
We introduce the main definitions
with relevant examples in Section 2.
In Section \ref{s:3} we will show that 
an action cost  $\Wcostname$   induces 
a canonical topology
$\Wtop$
on $\Xsp$, cf.\ Proposition 
\ref{prop:topology}. 
In fact, it even generates a \emph{uniform structure} on $\Xsp$, namely a topological structure (whose precise definition is postponed to Proposition \ref{prop:uniform-structure}  ahead) 
by means of which it is possible to render the concept that two points in $\Xsp$ are `close'
and to define an intrinsic notion of completeness. 
Since this uniform structure has a countable base, the associated topology
$\Wtop$
is metrizable. Indeed, 
in Section \ref{sec:metric} with Theorem \ref{thm:metrizability} we explicitly provide a family of 
equivalent metrics  $\Wmetricl\lambda$ 
metrizing the topology $\Wtop$
and inducing the uniform structure.
\par
As shown in \cite{PRST22}, the Minimizing Movement scheme \eqref{eq:time-incremental-intro} leads to limiting curves $u$ with finite 
$\Wcostname$-action on $[0,T]$, i.e.\ such that 
\[
\VarW u 0T: = \sup \left \{ \sum_{j=1}^M  
\Wcost{t^j - t^{j-1}}{u(t^{j-1})}{u(t^j)} \, : \ (t^j)_{j=0}^M \in \mathscr{P}_f([0,T])   \right\}<+\infty
\]
 (where $\mathscr{P}_f([0,T]) $ denotes the set of all partitions of $[0,T]$). 
In Section \ref{s:4} we focus on these curves and show that they indeed have $\BV$-like properties. In particular, 
when $(\Xsp,\Wcostname)$ is complete 
they are \emph{regulated} in the $\Wtop$-topology, hence their jump set is well defined. 
We then turn to \emph{$\Wcostname$-absolutely continuous} curves $u: [0,T]\to \Xsp$
in Section \ref{s:5}, for which an \emph{action density} $\smdn u{}$  can be introduced, fulfilling 
\[
\DVT{t{-}s}{u(s)}{u(t)} \leq \int_s^t \smd ur \dd r \qquad \text{for all }   0\leq s \leq t \leq T 
\]
and
\[
\VarW u 0T = \int_0^T \smd ur \dd r \,.
\]
In Section \ref{s:6} we provide a sufficient condition on the metric cost
ensuring that all finite-action curves on $[0,T]$  are in fact $\Wcostname$-absolutely continuous on $[0,T]$, cf.\ Theorem \ref{thm:AC1}.
Finally, in Section \ref{s:7} we  demonstrate that, under an additional condition on $\Wcostname$ which amounts to the existence of `approximate mid-points', a dynamic characterization for $\Wcostname$ is available, cf.\  Thm.\ \ref{thm:dynamic},  Thm.\ \ref{thm:optimality-BV} and Corollary \ref{cor:7.5}. In this way, we somehow `close the circle' by providing, for general costs $\Wcostname $, a \emph{dynamic interpretation} akin to \eqref{Finsler-cost-intro} for  action integral costs,  and to \eqref{dvt-intro} for Dynamical-Variational Transport costs. 
\medskip

\noindent
{\sl This paper is dedicated to Pierluigi Colli: it is a privilege  for us
to have him as a valuable  colleague and, more importantly, as  a loyal friend.}

%
%
%
%
\medskip

\noindent 

\section{Action spaces}
\label{s:example}
In this section we introduce
the main definitions we will deal with
and we show some important examples.
\begin{definition}[Action cost] 
\label{ass-W}
We say that a function
$
\Wcostname : (0,+\infty)\times \Xsp \times \Xsp \to [0,+\infty)
$
is an   action cost   on the set $\Xsp$ if it
satisfies the following  properties:
\begin{subequations}\label{assW}
\begin{enumerate}[label={(\arabic*)}]
	\item \label{tpc:1} \textbf{Strict positivity off the diagonal:} For all $\tau>0,\, \xv_0,\, \xv_1 \in \Xsp$,
	\begin{equation}\label{e:psi2}
		\Wcostname(\tau,\xv_0,\xv_1)= 0 \ \Leftrightarrow \ \xv_0=\xv_1.
		\end{equation}
  \item\label{tpc:1bis} \textbf{Symmetry:}
  For every $\tau>0$, $u_1,u_2\in \Xsp$
  \begin{equation}
\label{symmetry}
\Wcost{\tau}{\xv_1}{\xv_2} = \Wcost{\tau}{\xv_2}{\xv_1} \qquad \text{for all } \tau \in (0,+\infty), \, \xv_1,\, \xv_2 \in \Xsp\,. 
\end{equation}
	\item \label{tpc:2} \textbf{Concatenation inequality:} For all 
 $\xv_1,\, \xv_2,\,\xv_3\in\Xsp$  and   $\tau_{1}, \tau_{2}  \in (0,\pinfty)$
 	\begin{equation}\label{e:psi3}
 \Wcostname(\tau_{1}+\tau_{2}, \xv_1,\xv_3) \leq  \Wcostname(\tau_1,\xv_1,\xv_2) +  \Wcostname(\tau_2,\xv_2,\xv_3).
 	\end{equation}
\end{enumerate}
\end{subequations}
We call \emph{action space} a pair $(\Xsp,\Wcostname)$  consisting
of a set $\Xsp$ endowed with an action cost  $\Wcostname$.
\end{definition}
As  mentioned in the Introduction, 
$\Wcostname(\tau,u_0,u_1)$
represents the cost to reach $u_1$
from $u_0$ in the amount of time $\tau>0$.
\par
A first important
consequence of \eqref{e:psi2} and 
\eqref{e:psi3} is the monotonicity property
w.r.t.~$\tau$:
\begin{equation}
    \label{e:decreasing}
    0<\tau'<\tau''\quad\Rightarrow\quad
    \Wcostname(\tau',u_0,u_1)
    \ge \Wcostname(\tau'',u_0,u_1)\quad
    \text{for every }u_0,u_1\in \Xsp.
\end{equation}
In order to check \eqref{e:decreasing}
it is sufficient to notice that 
\begin{equation*}
   \Wcostname(\tau'',u_0,u_1)
   \le 
   \Wcostname(\tau',u_0,u_1)+
   \Wcostname(\tau''-\tau',u_1,u_1)
   =\Wcostname(\tau',u_0,u_1),
\end{equation*}
so that 
the map $\tau\mapsto \Wcostname(\tau,u_0,u_1)$
is decreasing.   Estimate 
\eqref{e:decreasing} renders the intuitive property that  the `cost for connecting' $u_0$ and $u_1$ decreases if
they are joined over  
 a longer time interval,
 and it is a consequence of 
 the positivity of $\Wcostname$
 and the fact that ``staying'' 
 at the same point is costless.
In particular, we can define
\begin{subequations}
    \label{eq:plusminus}
\begin{align}
    \Wcostplus \tau{u_0}{u_1}:={}&
    \inf_{\tau'<\tau}
    \Wcost{\tau'}{u_0}{u_1}
    =
    \lim_{\tau'\uparrow\tau} 
    \Wcost{\tau''}{u_0}{u_1},\\
    \Wcostminus \tau{u_0}{u_1}:={}&
    \sup_{\tau''>\tau}
    \Wcost{\tau''}{u_0}{u_1}
    =
    \lim_{\tau''\downarrow\tau} 
    \Wcost{\tau''}{u_0}{u_1},
    \end{align}
    \end{subequations}
    observing that 
    \[\Wcostminus\tau{u_0}{u_1}\le 
    \Wcost\tau{u_0}{u_1}\le 
    \Wcostplus\tau{u_0}{u_1}\quad\text{for every }\tau>0,\ u_0,u_1\in \Xsp.\] 
    It is easy to check that 
    \begin{proposition}
        \label{prop:maybe-useful}
        If $\Wcostname$ 
        is an action cost on $\Xsp$
        then the functions
        $\mathsf a_-$
        and $\mathsf a_+$
        defined by {\rm (\ref{eq:plusminus}a,b)}
        are action costs as well.
    \end{proposition}
    \begin{proof}
        Let us just check the concatenation inequality for $\mathsf a_+$, as the corresponding property for $\mathsf a_-$ follows by a similar argument.
        For $\tau_1,\tau_2>0$,
        $u_1,u_2,u_3\in \Xsp$,
        $\eps\in (0,\tau_1\land \tau_2)$ we have
        \begin{align*}
            \Wcost{\tau_1+\tau_2-2\eps}{u_1}{u_3}&\le 
            \Wcost{\tau_1-\eps}{u_1}{u_2}
            +
            \Wcost{\tau_2-\eps}{u_2}{u_3}.
        \end{align*}
        Passing to the limit as $\eps\down0$
        we obtain
        \begin{align*}
            \Wcostplus{\tau_1+\tau_2}{u_1}{u_3}&\le 
            \Wcostplus{\tau_1}{u_1}{u_2}
            +
            \Wcostplus{\tau_2}{u_2}{u_3}.
            \qedhere
        \end{align*}
    \end{proof}
\noindent    We also set
\begin{equation}
\label{a-sup}
    \Wcostsup{u_0}{u_1}:=
    \sup_{\tau>0}
    \Wcost\tau{u_0}{u_1}
    =
    \lim_{\tau\downarrow0}
    \Wcost\tau{u_0}{u_1}
    \in [0,+\infty]\,.
\end{equation}
\begin{definition}[Continuity and superlinearity]
\label{def:superlinearity}
We say that the  action cost 
\begin{itemize}
\item[-] $\Wcostname$ is \emph{continuous}
if for every $u_0,u_1\in \Xsp$ the map
$\tau\mapsto \Wcost \tau{u_0}{u_1}$ is continuous.
Equivalently, if
\begin{equation}
\label{e:continuity1}
    \Wcostplus\tau{u_0}{u_1}=
    \Wcostminus\tau{u_0}{u_1}
    \quad\text{for every }\tau>0,\ u_0,u_1\in \Xsp;
\end{equation}
\item[-] $\Wcostname$  is
\emph{metric-like} if
$\Wcostsup{u_0}{u_1}<+\infty$
for every $u_0,u_1\in \Xsp$;
\item[-]  $\Wcostname$
has a \emph{local superlinear growth}
if 
\begin{equation}
    \label{e:superlinear}
    \Wcostsup {u_0}{u_1}=+\infty
    \quad\text{for every }u_0,u_1\in\Xsp,\ 
    u_0\neq u_1.
\end{equation}
\end{itemize}
\end{definition}
\noindent
It is immediate to check that 
if $\Wcostname$ is metric-like
then the function
$\Wcostsupname$ 
is a metric in $\Xsp$.
 In turn, we refer to \eqref{e:continuity1}
as a continuity property on the grounds of Proposition \ref{prop:A-continuity} ahead. In what follows, 
we will mostly focus on the superlinear case \eqref{e:superlinear}.

\subsection{Examples}
\label{ss:2.1}
 We illustrate the above definitions in some examples. 
\begin{example}[Metrics]
\label{ex:trivial}
    If $d$ is a metric on $\Xsp$
    then the $\tau$-independent  cost 
    $\Wcost\tau uv:=d(u,v)$
    is an   action cost. 
    In fact,  an action cost 
    $\Wcostname$ is $\tau$-independent
    if and only if it is a metric on $\Xsp$.
\end{example}
\begin{example}[Rescaling]
    If $\mathsf b$ is an action cost
    on $\Xsp$ and $\lambda,\theta>0$
    then also the rescaled function
    \begin{equation}
        \Wcost \tau uv:=
        \theta\mathsf b({\tau/\lambda},u,v)
    \end{equation}
    is an action cost.
\end{example}
\begin{example}[The convex construction]
\label{ex:psi-construction}
\upshape
If $\Gacostname$ is  an action cost  on $\Xsp$ and
$ \Psidens: [0,+\infty) \to [0,+\infty) $ 
is a convex function with $0=\Psidens (0)<\Psidens(a)$ for every $a>0$, then
\begin{equation}
    \label{eq:psi-construction}
    \Wcost \tau uv:=
    \tau\Psidens\Big(
    \tau^{-1}\Gacost \tau uv\Big)
    \quad\text{is  an action cost.}
\end{equation}
It is immediate to check strict positivity and symmetry. 
The concatenation property
follows by the 
convexity and the monotonicity of $\Psidens$:
for every $\tau_i>0$, 
$i=1,2$,
with $\tau:=\tau_1+\tau_2$ and $\alpha_i:=\tau_i/\tau$,
we have
\begin{align*}
    \Wcostname(\tau,u,w)&=
    \tau \Psidens\Big(
    \tau^{-1}
    \Gacost{\tau_1+\tau_2}uw\Big)
    \\&\le 
    \tau \Psidens\Big(
    \tau^{-1}
    \big(\Gacost{\tau_1}uv {+} \Gacost{\tau_2}vw  \big)\Big)
    \\&=
    \tau \Psidens\Big(
    \alpha_1\tau_1^{-1}
    \Gacost{\tau_1}uv+
    \alpha_2\tau_2^{-1}\Gacost{\tau_2}vw\big)\Big)
    \\&\le 
    \alpha_1\tau \Psidens\Big(
    \tau_1^{-1}
    \Gacost{\tau_1}uv\Big)
    +
    \alpha_2\tau
    \Psidens\Big(\tau_2^{-1}\Gacost{\tau_2}vw\big)\Big)
    \\&= \Wcostname(\tau_1,u,v)+\Wcostname(\tau_2,v,w).
\end{align*}
\end{example}
\begin{example}[Action cost induced by a metric]
\label{ex:psi-metric}
    Recalling the metric case of Example 
\ref{ex:trivial}, 
 as a particular case of the construction set up in Ex.\ \ref{ex:psi-construction} 
we  get that  the functional 
\begin{equation} 
\label{eq:psi-example}
\text{$\Wcost \tau uv= \tau \Psidens \left( \frac{d(u,v)}{\tau}\right)$, for a given metric $d$, on $\Xsp$
is  an  action cost.} 
\end{equation}
Concerning the properties stated in 
Definition \ref{def:superlinearity}, we 
immediately see that the  action cost 
defined by \eqref{eq:psi-example} is continuous;
moreover, setting
\begin{equation}
    \label{eq:recession}
    \Psidens_\infty':=
    \lim_{r\to+\infty}\frac{\Psidens(r)}r
    =
    \sup_{r>0}\frac{\Psidens(r)-\Psidens(r_0)}{r-r_0}
    \in (0,+\infty]
\end{equation}
we have two cases:
\begin{enumerate}
\item If $\Psidens_\infty'<+\infty$,
then $\Wcostname $ is metric-like;
\item If  $\Psidens_\infty'=+\infty$ 
(i.e.~$\Psidens$ has superlinear growth at infinity), then  $\Wcostname$
has a local superlinear growth.
\end{enumerate}
\end{example}
\begin{example}[Linear combination of  action costs] 
\label{ex:2.8}
    It is immediate to check that 
    if $\Wcostname_i$, $i$ running in a finite set 
    $ \mathcal I$,
    are   action costs  on $\Xsp$
    and $\theta_i>0$ are positive real numbers, 
    then also
    $\Wcostname:=\sum_i \theta_i\Wcostname_i$
    is an  action cost. 

    In particular,
    if $\Xsp$ is endowed with two 
    metrics $\metrname{1}$ and $\metrname{2}$ 
we may consider the  action cost 
\begin{equation}
\label{structure-W-1+2}
\Wcost \tau uw = \tau \psidens{1}\left( \frac{\metr{1}uw}\tau \right)+ \tau \psidens{2}\left( \frac{\metr{2}uw}\tau \right) 
\end{equation}
where $\psidens i$ are as in Example
\ref{ex:psi-metric}.
%
  We mention that
 an action cost  induced by two metrics as in  \eqref{structure-W-1+2}
 occurs in the Minimizing-Movement scheme
for the generalized gradient system $(\Xsp, \metrname 1, \psidens 1, \metrname 2, \psidens 2)$ providing the 
 vanishing-viscosity approximation of the  rate-independent system $ (\Xsp, \metrname 1, \psidens 1)$,  cf.\ e.g.\ \cite{MRS09,MRS13}.
\end{example}
\begin{example}[Concave compositions]
\label{ex:concave}
    Let $h:(0,+\infty)\times [0,+\infty)^{I}\to [0,+\infty)$
    be a concave function 
    such that $h(\tau,0)=0$ 
    for every $\tau>0$ 
    and $h(\tau,\boldsymbol a)>0$
    for every $\tau>0,\boldsymbol a\neq 0.$
    If $\Wcostname_i$, $i=1,\cdots, I,$
    are  action costs  on $\Xsp$ then also
    \begin{equation}
        \label{eq:concave-construction}
        \Wcostname(\tau,u,v) :=h(\tau,\Wcostname_1(\tau,u,v),\cdots,\Wcostname_I(\tau,u,v))\quad
        \tau>0,\ u,v\in \Xsp
    \end{equation}
    is an  action cost.
    We just check the concatenation property,
    by using the facts that 
    \[0\le a_i\le a_i',\ i=1,\cdots,I\quad\Rightarrow\quad
    h(\tau, a_1,\cdots, a_I)\le h(\tau, a_1',\cdots, a_I')\]
    \[h(\tau'+\tau'',a_1+b_1,\cdots,a_I+b_I)\le 
    h(\tau',a_1,\cdots,a_I)+h(\tau'',b_1,\cdots,b_I).\]
    If $\tau=\tau_1+\tau_2$ we have
    \begin{align*}
        \Wcost\tau uw&=
        h(\tau,\Wcostname_1(\tau,u,w),
        \cdots,\Wcostname_I(\tau,u,w))
        \\&\le 
        h(\tau_1+\tau_2,
        \Wcostname_1(\tau_1,u,v)+
        \Wcostname_1(\tau_2,v,w),
        \cdots,\Wcostname_I(\tau_1,u,v)+
        \Wcostname_I(\tau_2,v,w))
        \\&\le 
        h(\tau_1,\Wcostname_1(\tau_1,u,v),
        \cdots,\Wcostname_I(\tau_1,u,v))
        +
        h(\tau_2,\Wcostname_1(\tau_2,v,w),
        \cdots,
        \Wcostname_I(\tau_2,v,w))
        \\&=
        \Wcost{\tau_1}uv+
        \Wcost{\tau_2}vw.
    \end{align*}
\end{example}
\begin{example}[Supremum of a directed family]
\label{ex:supremum-directed}
    Let $\mathcal A$ 
    be a directed family of 
    actions costs on $\Xsp$, i.e.~
    for every $\mathsf a_1,\mathsf a_2\in \mathcal A$ there exists $\mathsf a\in \mathcal A$ such that 
    $\mathsf a_1\lor \mathsf a_2\le \mathsf a.$
    If 
    \begin{equation}
        \bar{\mathsf a}:=\sup_{\mathsf a\in \mathcal A}\mathsf a
    \end{equation}
    is finite in $(0,+\infty)\times \Xsp\times \Xsp$, then $\bar{\mathsf a}$ is an action cost as well.
\end{example}
\begin{example}[Supremum of truncated metrics]
\label{ex:supremum-metrics} 
    Let $\mathsf d_\lambda$, $\lambda>0$, 
    be a family of
    metrics on $\Xsp$, increasing w.r.t.~$\lambda$,
    such that $
    \sup_{\lambda>0}\mathsf d_\lambda(u,v)<\infty $
    for every $u,v\in \Xsp.$
    Then
    \begin{equation}
    \label{eq:general}
    \Wcost \tau uv:=\sup_{\lambda>0}
    \mathsf d_\lambda(u,v)\land \lambda \tau
    \end{equation}
    is an action cost.
    In fact, each term $\mathsf a_\lambda:=
    \mathsf d_\lambda\land \lambda \tau$
    is an action cost, thanks 
    Examples \ref{ex:trivial} and 
    \ref{ex:concave}. Moreover, the
    set $\mathcal A:=\{\mathsf a_\lambda\}_{\lambda>0}$
    is obviously directed: 
    by Example \ref{ex:supremum-directed}
    it is sufficient to prove that 
    $\mathsf a$ is finite.
\end{example}
 \subsection{Action integral costs}
 \label{ss:Finsler}
 Let $\Xsp$ be, in addition, a separable and reflexive Banach space with norm $\|\cdot\|$, and let us consider an \emph{integrand} $\mathcal{R}: \Xsp\times\Xsp\to [0,+\infty)$ 
  which is bounded on bounded sets and fulfils 
  the following properties:
 \begin{itemize}
     \item[$(\mathcal{R}_1)$]  for all 
     $(\teta_n)_n,\, (\zeta_n)_n \subset \Xsp$
     \begin{equation}
         \label{R1-prop}
         \left.
         \begin{array}{ccc}
              \teta_n &  \weakto  & \teta,
              \\
                \zeta_n &  \weakto  & \zeta,
         \end{array}
         \right\}
         \ \Rightarrow \ \liminf_{n\to\infty}
         \mathcal{R}(\teta_n,\zeta_n) \geq \mathcal{R}(\teta,\zeta);
     \end{equation}
       \item[$(\mathcal{R}_2)$] 
       for all $\teta\in \Xsp$ the functional $\mathcal{R}(\teta,\cdot)$ is convex, even, and 
       \begin{equation}
         \label{R2-prop}
       \mathcal{R}(\teta,\zeta) = 0 \ \text{ if and only if } \ \zeta =0;
       \end{equation}
        \item[$(\mathcal{R}_3)$] 
        there exists $\Phi_{\mathcal{R}}: \Xsp \to [0,+\infty)$ with $\lim_{\|\zeta\|\up +\infty} \frac{\Phi_{\mathcal{R}}(\zeta)}{\|\zeta\|} =+\infty$ such that 
        \begin{equation}
         \label{R3-prop}
     \forall\, (\teta,\zeta) \in \Xsp\times \Xsp\,: \quad   \mathcal{R}(\teta,\zeta) \geq \Phi_{\mathcal{R}}(\zeta)\,.
       \end{equation}
  \end{itemize}
  For later use, we point out that property $(\mathcal{R}_3) $ is indeed equivalent to the existence of a \emph{convex} and increasing function $\phi_{\mathcal{R}}: [0,+\infty) \to [0,+\infty)$ such that 
  \begin{equation}
  \label{later-use}
     \forall\, (\teta,\zeta) \in \Xsp\times \Xsp\,: \quad   \mathcal{R}(\teta,\zeta) \geq  \phi_{\mathcal{R}}(\|\zeta\|)\,.
  \end{equation}
  \par
 Relying on properties $(\mathcal{R}_1) $-- 
 $(\mathcal{R}_3) $  we are in a position to prove the following result.
  \begin{proposition}
  \label{prop:action-int-props}
  Let $\Wcostname: (0,+\infty)\times \Xsp \times \Xsp \to [0,+\infty)$ be defined by
  \begin{equation}
 \label{Finsler-cost}
 \Wcostname(\tau,u,v): = \inf \left\{ \int_0^{\tau} \finsl{\Theta(r)}{\Theta'(r)} \dd r \, : \ \Theta \in \AC([0,\tau];\Xsp)\,, \ \Theta(0)=u, \ \Theta(\tau)= v  \right\}.
 \end{equation} 
 Then, 
  the infimum in \eqref{Finsler-cost} is attained, 
 $\Wcostname$ is a continuous action cost, with local superlinear growth.
  \end{proposition}
\begin{proof}
It is easy to check that $\Wcostname$   is an action cost in the sense of Def.\ \ref{ass-W}. We now show that the $\inf$ in \eqref{Finsler-cost} is attained for all $(\tau,u,v) \in (0,+\infty)\times \Xsp\times \Xsp$. Indeed, let  $(\Theta_n)_n$ be a  minimizing sequence: thanks to   \eqref{R3-prop} we have 
$
\sup_n \int_0^\tau \Phi_{\mathcal{R}}(\Theta_n'(r)) \dd r \leq C $.
Combining this with the fact that $\Phi_{\mathcal{R}}$ has superlinear growth, and taking into account that $\Theta_n(0)=u$ and $\Theta_n(1)=v$, we conclude that
\begin{equation}
    \label{estimates}
(\Theta_n)_n \text{ is bounded in } L^\infty(0,\tau;\Xsp) \text{ and $(\Theta_n')_n $ is uniformly integrable in }
L^1(0,\tau;\Xsp).
\end{equation}
Therefore, there exist a  (not relabeled) subsequence and $\Theta \in \AC ([0,\tau];\Xsp)$ such that 
\[
\begin{cases}
\Theta_n \weakto \Theta &  \text{ in } W^{1,1}(0,\tau;\Xsp),
\\
\Theta_n(t) \weakto \Theta(t) & \text{ in } \Xsp \text{ for all $t \in [0,T],$}
\end{cases}
\]
so that $\Theta$ connects $u$ to $v$. 
By a variant 
(cf.\ \cite[Theorem 21]{Valadier90})
of the Ioffe Theorem we gather that 
\[
 \Wcostname(\tau,u,v)=\liminf_{n\to\infty} \int_0^\tau \finsl{\Theta_n(r)}{\Theta_n'(r)} \dd r \geq \int_0^\tau \finsl{\Theta(r)}{\Theta'(r)} \dd r \,.
\]
\par
In order to prove continuity of $\Wcost{\cdot}uv$ with fixed $u,\,v \in \Xsp$,
let us take $\tau>0$ and a sequence $\tau_n \downarrow \tau$; let 
$\Theta_n$ and $\Theta$ optimal curves for 
 $\Wcostname(\tau_n,u,v)$ and $  \Wcostname(\tau,u,v) $, respectively. Extend $\Theta$ to a curve $\overline\Theta$ on $[0,\tau_n]$  by setting $\overline\Theta(t): = \Theta(\tau) = v$ for all $t \in (\tau,\tau_n]$. Then, $\overline\Theta$  is an admissible competitor for the minimum problem defining $\Wcostname(\tau_n,u,v)$, and we thus have for all $n\in \N$
 \[
 \Wcostname(\tau_n,u,v) \leq \int_0^{\tau_n} \finsl{\overline{\Theta}(r)}{\overline{\Theta}'(r)} \dd r
  \stackrel{(1)}=  \int_0^{\tau} \finsl{\overline{\Theta}(r)}{\overline{\Theta}'(r)} \dd r =  \Wcostname(\tau,u,v)\,,
 \]
where {\footnotesize (1)} follows from the fact that $\overline{\Theta}' \equiv 0$ on $[\tau,\tau_n]$. 
Therefore,
\begin{equation}
    \label{limsup-inequality}
\limsup_{n\to\infty}  \Wcostname(\tau_n,u,v) \leq\Wcostname(\tau,u,v)\,.
\end{equation}
We now aim to show 
\begin{equation}
    \label{liminf-inequality}
\liminf_{n\to\infty}  \Wcostname(\tau_n,u,v) 
= \liminf_{n\to\infty} \int_0^{\tau_n} \finsl{\Theta_n(r)}{\Theta_n'(r)} \dd r
\geq\Wcostname(\tau,u,v)
\end{equation}
Let $(\Theta_n)_n$ be a (non-relabeled) subsequence for which the above $\liminf$ is a $\lim$. It follows from \eqref{limsup-inequality}
that for $(\Theta_n)_n$  estimates \eqref{estimates} hold. In particular, from the uniform integrability of $(\Theta_n')_n$ we gather that 
\[
\forall\, \epsilon>0 \ \exists\, \bar n \in \N\ \forall\, n \geq \bar n \, : \|\Theta_n(\tau_n){-}\Theta_n(\tau)\| \leq \int_{\tau}^{\tau_n} \|\Theta_n'(r) \| \dd r \leq \epsilon\,.
\]
Choosing $\epsilon = \frac1k$, we thus extract  subsequences $(\tau_{n_k})_k$
 and $(\Theta_{n_k})_k$
such that for all $k\geq 1$
\[
\|\Theta_{n_k}(\tau_{n_k}){-}\Theta_{n_k}(\tau)\| \leq  \frac1{k}\,.
\]
Now, the same compactness arguments as in the above lines apply to the sequence $(\Theta_{n_k}|_{[0,\tau]})_k$, yielding convergence, along a non-relabeled subsequence, to a curve $\widehat\Theta \in \AC ([0,\tau];\Xsp)$  connecting $u$ to $v$, since
\[
\widehat\Theta(\tau) = \lim_{k\to\infty} \Theta_{n_k}(\tau) = \lim_{k\to\infty} \Theta_{n_k}(\tau_{n_k}) = v\,.
\]
 Therefore, we 
 gather that 
 \[
\Wcost \tau uv \leq  \int_0^{\tau} \finsl{\widehat\Theta(r)}{\widehat\Theta'(r)} \dd r \leq \liminf_{k\to\infty} 
\int_0^{\tau_{n_k}} \finsl{\Theta_{n_k}(r)}{\Theta_{n_k}'(r)} \dd r = \liminf_{n\to\infty}  \Wcostname(\tau_n,u,v)\,,
 \]
 and \eqref{liminf-inequality} ensues. 
 All in all, we have shown that 
 $\lim_{n\to\infty}\Wcost{\tau_n}uv = \Wcost{\tau}uv  $ whenever $\tau_n\down \tau$. With analogous arguments, we show that
  $\lim_{n\to\infty}\Wcost{\tau_n}uv = \Wcost{\tau}uv $ if $\tau_n\up \tau$.
  \par
  In order to show that  
  \begin{equation}
\label{local-sup-growth}   
\Wcostsup{u}{v}=+\infty \qquad \text{if $u \neq v$,}
\end{equation}
let us observe that, for any fixed $\tau>0$, with  $\Theta$ an optimal curve for $\Wcost \tau uv$, the following estimates hold:
  \[
  \begin{aligned}
\Wcost \tau uv = \tau \dashint_0^\tau 
\finsl{\Theta(t)}{\Theta'(t)} \dd t \stackrel{(1)}\geq 
\tau  \dashint_0^\tau \phi_{\mathcal{R}}(\|\Theta'(t)\|) \dd t 
& \stackrel{(2)}\geq \tau \phi_{\mathcal{R}}
\left(\left\| \dashint_0^\tau \Theta'(t) \dd t \right\|\right)
\\ & 
= \tau \phi_{\mathcal{R}}\left(\frac{\|v{-}u\|}\tau \right)
\end{aligned}
  \]
  where {\footnotesize (1)} ensues from \eqref{later-use}, and {\footnotesize (2)} from Jensen's inequality, as  $\phi_{\mathcal{R}}$ is  convex. 
  Property \eqref{local-sup-growth} then follows, taking into account that $\phi_{\mathcal{R}}$ has superlinear growth at infinity.  
\end{proof}

 \bigskip
 
 

\section{The topology induced by 
an  action cost}
\label{s:3}

Our first step is to show that 
an  action cost  $\Wcostname$
induces a natural Hausdorff 
topology in $\Xsp$ satisfying the first countability axiom.
We start by introducing 
a fundamental system of neighborhoods of 
every $\xv\in \Xsp$: it is 
the collection of sets
$\Nei \xv\tau c$ indexed by the real
parameters $\tau,c>0$
\[
\Nei \xv\tau c: = \Big\{ v \in \Xsp\, : \Wcost \tau \xv v <c \Big\} \quad \text{for } \xv \in \Xsp,  \ \tau\in (0,+\infty), \ c\in (0,+\infty)\,.
\]
\begin{proposition}
\label{prop:topology}
  Let $(\Xsp,\Wcostname)$ be an action space
according to Definition \ref{ass-W} \nc 
and consider for every $\xv\in \Xsp$ 
the family of sets 
\begin{equation}
\label{e:base}
    \mathscr U(\xv):=\Big\{\mathrm U\subset \Xsp:\mathrm U\supset \Nei \xv\tau c
\text{ for some }\tau,c>0\Big\}.
\end{equation}
The collection 
$\mathscr U(\xv)$ satisfies the 
axioms of a (Hausdorff) 
neighborhood system, 
i.e.
\begin{enumerate}
\item If $\mathrm U\in \mathscr U(\xv)$
and $\mathrm U\subset \mathrm U'$ then
$\mathrm U'\in \mathscr U(\xv)$;
\item Every finite intersection of 
elements of $\mathscr U(\xv)$ belongs to $\mathscr U(\xv)$;
    \item 
    The element $\xv$ belongs to  every 
    set of $\mathscr U(\xv)$;
    \item If $\mathrm U\in \mathscr U(\xv)$
    then there is $\mathrm V\in \mathscr U(\xv)$
    such that $\mathrm U\in \mathscr U(v)$ for every $v\in \mathrm V$;
    \item If $\xv_1\neq \xv_2$ 
    then there exist 
    $\mathrm U_i\in \mathscr U(\xv_i)$, 
    $i=1,2$, such that 
    $\mathrm U_1\cap\mathrm U_2=\varnothing$.
\end{enumerate}
In particular, there exists a unique 
topology $\Wtop$ such that 
for every $\xv\in \Xsp$ 
$\mathscr U(\xv)$
is the system of neighborhoods
in the topology $\Wtop$.  Moreover, 
$\Wtop$ is a Hausdorff 
topology satisfying the first countability axiom. 
\end{proposition}
\begin{proof}
Property (1): obvious by the definition
of \eqref{e:base}.

\medskip\noindent
Property (2): if $\mathrm U_n\in 
\mathscr U(\xv)$, $n=1,\cdots, N$,
then we can find
$\tau_n,c_n>0$, $n=1,\cdots N$, such that 
$\mathrm U_n\supset \Nei\xv{\tau_n}{c_n}$.
Setting $\tau:=\max_n \tau_n$ and
$c_n:=\min_n c_n$,
\eqref{e:decreasing} shows that 
$\Nei\xv\tau c\subset \mathrm U_n$
for every $n$, so that 
$\bigcap_{n=1}^N \mathrm U_n\in \mathscr U(\xv)$.

\medskip\noindent
Property (3) 
is obvious since
$ \Wcost \tau {\xv}{\xv} =0
$
for all $\tau>0$.

\medskip\noindent
Property (4): 
since $\mathrm U\in \mathscr U(\xv)$ 
there exists $\tau,c>0$ 
such that 
$\Nei\xv\tau c\subset \mathrm U$.
We then define
$\mathrm V:=\Nei \xv{\tau/2}{c/2}$
and we observe that 
for every $v\in \mathrm V$
and every $z\in \Nei v{\tau/2}{c/2}$
\begin{equation*}
    \Wcost\tau z\xv
    \le 
    \Wcost{\tau/2} zv+
    \Wcost{\tau/2}v\xv
    <c/2+c/2=c
\end{equation*}
so that $z\in \mathrm U$.
Therefore $\mathrm U\supset \Nei v{\tau/2}{c/2}$,   whence $\mathrm{U} \in  \mathscr U(v)$. 

\medskip\noindent
Property (5):
For some $\tau>0$ we define
$c: = \Wcost{\tau}{\xv_1}{\xv_2}>0$
and we set $\mathrm U_i:=
\Nei{\xv_i}{\tau/2}{c/2}.$
Clearly 
$\mathrm U_1\cap \mathrm U_2=\emptyset$,
since 
otherwise, there would exist $y$ such that $\Wcost{\tfrac \tau 2}{\xv_i}y < c/2$
 for $i=1,2$, and thus by the triangle inequality
 \[
 \Wcost{\tau}{\xv_1}{\xv_2} \leq \Wcost{\tfrac \tau 2}{\xv_1}y + \Wcost{\tfrac \tau 2}y{\xv_2} < c  =  \Wcost{\tau}{\xv_1}{\xv_2}\,,
 \]
which is a contradiction.

In order to show that $\Wtop$ satisfies
the first countability axiom, it is 
sufficient to observe that for every $\xv\in \Xsp$
the sets
$
 (\Nei\xv{2^{-n}}{2^{-n}})_{n\in \N} $
form  a countable fundamental system of neighborhoods.
\end{proof}
\begin{remark}
\upshape
\label{unique-parameter}
The last statement in Proposition \ref{prop:topology}, and in particular the fact that sets $
 (\Nei\xv{2^{-n}}{2^{-n}})_{n\in \N} $, parametrized by the sole index $n$, provide a fundamental system of neighborhoods,
 suggest that, to generate the topology $\Wtop$,  it would be sufficient to work with a family of neighborhoods parametrized by a single 
 real parameter.
  Interestingly, 
 this approach is equivalent to the one 
 in which the topology is 
 generated by the sets $\Nei \xv\tau c$,
 which naturally arise from
 an   action cost. 
\end{remark} 
\par
We can considerably refine the previous
Proposition by showing, more or less with the same argument, that 
$\Wcostname$  even  induces a \emph{uniform structure} on $\Xsp$ (see
\cite[Chapter II, \S 1]{Bourbaki1}), which allows us to formalize the concept of  `closeness' of two points. More precisely, 
for every $\tau,c>0$ we define
\begin{equation}
    \label{eq:uniform}
    \VNei{\tau}{c}:=
    \Big\{(\xv_1,\xv_2)\in \Xsp{\times}\Xsp\,: \ 
    \Wcost\tau{\xv_1}{\xv_2}<c\Big\}.
\end{equation}
We will also use the following notation
for subsets $\mathrm V,\mathrm V_i$ of $\Xsp\times \Xsp:$
\begin{equation}
    \mathrm V^{-1}:=\Big\{(\xv_2,\xv_1):
    (\xv_1,\xv_2)\in \mathrm V\Big\}
\end{equation}
\begin{equation}
    \mathrm V_2\circ\mathrm V_1:=
    \Big\{(\xv_1,\xv_3)\in \Xsp{\times}\Xsp\,:
    \exists\, \xv_2\in \Xsp\text{ such that }
    (\xv_1,\xv_2)\in \mathrm V_1,\ 
    (\xv_2,\xv_3)\in \mathrm V_2\Big\}.
\end{equation}
\begin{proposition}
\label{prop:uniform-structure}
    Let 
 $(\Xsp,\Wcostname)$ 
 be an action space.
The family of sets 
\begin{equation}
\label{e:base-uniform}
    \mathfrak U:=\Big\{\mathrm U\subset \Xsp\times\Xsp:\mathrm U\supset \VNei \tau c
\text{ for some }\tau,c>0\Big\}.
\end{equation}
 satisfies the 
axioms of a uniform structure,
i.e.
\begin{enumerate}
\item If $\mathrm U\in \mathfrak U$
and $\mathrm U\subset \mathrm U'$ then
$\mathrm U'\in \mathfrak U$;
\item Every finite intersection of 
elements of $\mathfrak U$ 
belongs to $\mathfrak U$;
\item Every $\mathrm U\in \mathfrak U$
contains the 
diagonal $\Delta:=\Big\{(\xv,\xv):\xv\in \Xsp\Big\}$ in $\Xsp\times \Xsp$;
\item If $\mathrm U\in \mathfrak U$
then also $\mathrm U^{-1}$
belongs to $\mathfrak U$;
    \item if $\mathrm U\in \mathfrak U$
    then there is $\mathrm V\in \mathfrak U$
    such that $\mathrm V\circ\mathrm V
    \subset \mathrm U$. 
\end{enumerate}
\end{proposition}
\begin{proof}
    The proof of properties (1), (2), (3)
    follows  by the same arguments for
    the corresponding properties
    of Proposition \ref{prop:topology}.
    
    Property (4) is an immediate consequence
    of \eqref{symmetry} and
    the symmetry of 
    $\Wcostname$, yielding that 
     $(u_1,u_2) \in \VNei{\tau}{c}$ if and only if  $(u_2,u_1) \in \VNei{\tau}{c}$.

    In order to prove property (5)
    we argue as for claim (4) of Proposition
    \ref{prop:topology}:
    we first select $\tau,c>0$
    such that $\VNei\tau c\subset \mathrm U$
    and we set 
    $\mathrm V:=\VNei{\tau/2}{c/2}$,
    observing that 
    if $(\xv_1,\xv_3)\in \mathrm V\circ\mathrm V$
    we can find $\xv_2$ such that 
    $\Wcost{\tau/2}{\xv_1}{\xv_2}<c/2$
    and 
    $\Wcost{\tau/2}{\xv_2}{\xv_3}<c/2$.
    Applying \eqref{e:psi3}
    we get
    $\Wcost{\tau}{\xv_1}{\xv_2}<c$,
    i.e.~$(\xv_1,\xv_3)\in \mathrm U.$
%
\end{proof}
\par Following the terminology of \cite{Bourbaki1}, we shall refer to the sets of $\mathfrak{U}$ as \emph{entourages}; if $\mathrm{V} $ is an entourage
in $\mathfrak{U}$ and $(\xv, \xv') \in \mathrm{V}$, we may say that $\xv$ and $\xv'$ are \emph{`$\mathrm{V}$-close'}. Likewise,  the structure $\mathfrak{U}$ induces a topology 
on $\Xsp$ such that the neighborhoods 
    of a point $\xv\in \Xsp$ are given by
    the sets 
    \begin{equation}
    \label{topology-structure}
    \mathrm V(\xv):=
    \{y\in \Xsp\,: \  (\xv,y)\in \mathrm U\}
    \qquad \text{for some $\mathrm U\in \mathfrak U$}.
    \end{equation}
    It is straightforward to check that 
 \begin{proposition}
 \label{prop:coincidence-topologies}
 The  
topology $\Wtop$ coincides with 
the topology induced by 
the uniform structure $\mathfrak U$.
 \end{proposition}
\par
Having a uniform structure
at our disposal, we can 
define a corresponding notion of 
Cauchy sequence and completeness.
\begin{definition}[Cauchy sequences
and complete action spaces]
    Let 
    $(\Xsp,\Wcostname)$ be an action space.
     We say that 
     a sequence $(\xv_n)_n\subset \Xsp$  is  a \emph{Cauchy sequence}  
     in the $\mathfrak U$-uniform structure 
if it enjoys the following property:
\begin{equation}
\label{Cauchy}
\forall\, \tau,c>0 \quad \exists\, \bar n \in \N:\quad 
\forall\, n,\, m \geq \bar n\,: \quad  (\xv_n,\xv_m) \in \TuNei{\tau}{c}.
\end{equation}
    We say that $(\Xsp,\Wcostname)$ is
    complete if 
    every Cauchy sequence is convergent.
\end{definition}

\par
As in Proposition
\ref{prop:topology}
we immediately see that 
the filter of the entourages
$\mathfrak U$ of the uniform structure
has a countable base, given by 
the collection 
$\VNei{2^{-n}}{2^{-n}}$.
We can thus apply
\cite[Theorem 1,  Chap. IX,  \S 2]{Bourbaki2}
to obtain the metrizability of $\mathfrak U$. 
In our setting, we can be even more precise
by introducing a metric which is induced by 
$\Wcostname$ and induces the 
same uniform and topological structure:
we will discuss this issue in the next section.
\section{Metric structures induced by 
an action cost}
\label{sec:metric}
\nc
\begin{definition}
    \label{def:Wmetric}
    For every $\xv_1,\xv_2\in \Xsp$
     and every $\lambda>0$ 
    we set
    \begin{align}
    \label{e:Wmetric-def}
        \Wmetricl\lambda(\xv_1,\xv_2):={}&
        \lambda    \inf\Big\{r\ge0:
        \Wcost r{\xv_1}{\xv_2}\le \lambda r\Big\}
        \\={}&
        \label{e:Wmetric-def3}
        \inf\Big\{s\ge0:
        \Wcost {s/\lambda}{\xv_1}{\xv_2}\le s\Big\}.
    \end{align}
    In the particular case
    $\lambda=1$ we also set 
    $\Wmetric:=\Wmetricl1$.
\end{definition}
 Recalling 
\eqref{e:decreasing} 
we have
\[0<\lambda'<\lambda''
\quad\Rightarrow\quad
\Big\{s\ge0:
        \Wcost {s/\lambda''}{\xv_1}{\xv_2}\le s\Big\}
        \subset
     \Big\{s\ge0:
        \Wcost {s/\lambda'}{\xv_1}{\xv_2}\le s\Big\}. 
\]
It is then immediate to check from
\eqref{e:Wmetric-def3} 
that 
the family $\Wmetricl\lambda$ is increasing w.r.t.~$\lambda$:
\begin{equation}
    \label{eq:metric-monotonicity}
    0<\lambda'<\lambda''
    \quad\Rightarrow\quad
    \Wmetricl{\lambda'}\le 
    \Wmetricl{\lambda''}\,.
\end{equation}
On the other hand, 
using \eqref{e:Wmetric-def},
we see that 
$\lambda^{-1}\Wmetricl\lambda$ is decreasing w.r.t.~$\lambda$:
\begin{equation}
    \label{eq:metric-monotonicity2}
    0<\lambda'<\lambda''
    \quad\Rightarrow\quad
    \frac{\Wmetricl{\lambda''}}{\lambda''}\le 
    \frac{\Wmetricl{\lambda'}}{\lambda'}\,,
\end{equation}
so that the metrics $\Wmetricl\lambda$
are equivalent and there holds
\begin{equation}
    \label{eq:metric-equivalence}
     (1\land \lambda)\Wmetric
     \le \Wmetricl{\lambda}\le 
    (1\lor \lambda)\Wmetric.
\end{equation}
\begin{lemma}[Equivalent characterizations]
For every $\lambda>0$ we 
have 
\begin{equation}
    \label{e:upperd}
    \Wmetricl\lambda(\xv_1,\xv_2)\le \lambda r
    \quad\Leftrightarrow\quad
    \Wcostminus{r}{\xv_1}{\xv_2}\le \lambda r
    \quad\Leftrightarrow\quad 
    \Wcost{\tau}{\xv_1}{\xv_2}\le \lambda r
        \quad\text{for every }\tau>r,
\end{equation}
and
\begin{equation}
    \label{e:lowerd}
    \Wmetricl\lambda(\xv_1,\xv_2)\ge \lambda r
    \quad\Leftrightarrow\quad
    \Wcostplus{r}{\xv_1}{\xv_2}\ge \lambda r
    \quad\Leftrightarrow\quad
    \Wcost{\tau}{\xv_1}{\xv_2}\ge \lambda r
        \quad\text{for every }\tau<r.
\end{equation}
In particular, 
$\Wmetricl\lambda$ can also be characterized by 
\begin{equation}
    \label{e:equivalent}
     \Wmetricl\lambda(\xv_1,\xv_2)=
    \lambda r
    \quad\Leftrightarrow\quad
    \Wcostminus r{\xv_1}{\xv_2}
    \le \lambda r\le \Wcostplus r{\xv_1}{\xv_2},
\end{equation}
and by the variational formulae
\begin{align}
\label{eq:variational1}
    \Wmetricl\lambda(u_1,u_2)&=
    \min_{\tau>0}\Wcostminus\tau{u_1}{u_2}\lor
    \lambda\tau,\\
\label{eq:variational2}
    \Wmetricl\lambda(u_1,u_2)&=
    \max_{\tau>0}\Wcostplus\tau{u_1}{u_2}\land
    \lambda\tau.
\end{align}
\end{lemma}
\begin{proof}
To check \eqref{e:upperd}
we observe that 
if $\Wmetricl\lambda(\xv_1,\xv_2)\le \lambda r$
then for every
 $r'$ with 
$\tau>r'>r$ we get
$\Wcost\tau{\xv_1}{\xv_2}\le 
\Wcost{r'}{\xv_1}{\xv_2}\le \lambda r'$
so that 
$\Wcost\tau{\xv_1}{\xv_2}\le \lambda r$.
Conversely,
if 
for every $\tau> r$ we have
$\Wcost\tau{\xv_1}{\xv_2}\le \lambda r$ 
we also get 
$\Wcost\tau{\xv_1}{\xv_2}\le \lambda \tau$, 
so that 
$\Wmetricl\lambda(\xv_1,\xv_2)\le \lambda \tau$
and eventually 
$\Wmetricl\lambda(\xv_1,\xv_2)\le \lambda r.$

Concerning characterization \eqref{e:lowerd},
if $\Wmetricl\lambda(\xv_1,\xv_2)\ge \lambda r$
then 
for $\tau<r'<r$ 
we  have 
$\Wcost\tau{\xv_1}{\xv_2}\ge
\Wcost{r'}{\xv_1}{\xv_2}>\lambda r'$,
so that 
$\Wcost\tau{\xv_1}{\xv_2}\ge \lambda r$.
Conversely,
if for $\tau<r$ we have
$\Wcost\tau{\xv_1}{\xv_2}\ge \lambda r>\lambda \tau$, then
we conclude that 
$\Wmetricl\lambda(\xv_1,\xv_2)\ge\lambda \tau$:
since $\tau<r$ is arbitrary we eventually get
$\Wmetricl\lambda(\xv_1,\xv_2)\ge \lambda r$.

Let us now check 
    \eqref{eq:variational1} (notice that the minimum in \eqref{eq:variational1} is attained since $\mathsf a_-$ is decreasing
    and lower semicontinuous).
    If $r:=\Wcostminus{\tau}{\xv_1}{\xv_2}\le \lambda \tau$
    then clearly $\Wmetricl\lambda(\xv_1,\xv_2)\le\lambda \tau$ by \eqref{e:upperd}.
    If $r>\lambda \tau$ then $\Wcostminus{r/\lambda}{\xv_1}{\xv_2}\le r=\lambda (r/\lambda)$ since
    $\tau\mapsto 
    \Wcostminus\tau{\xv_1}{\xv_1}$ is decreasing, 
    so that $\Wmetricl\lambda(\xv_1,\xv_2)\le r.$
    This argument shows 
    that 
    $\Wmetricl\lambda(\xv_1,\xv_2)\le
    \min_{\tau>0}\Wcostminus\tau{\xv_1}{\xv_2}\lor \lambda \tau$.
    On the other hand \eqref{e:equivalent}
    yields that, when $\xv_1\neq\xv_2$, 
    setting 
    $\tau:=\Wmetricl\lambda(\xv_1,\xv_2)/\lambda$
    we have 
    $\Wmetricl\lambda(\xv_1,\xv_2)=
    \Wcostminus\tau{\xv_1}{\xv_2}\lor \lambda\tau$, which provides the equality in \eqref{eq:variational1}.

    A similar argument yields \eqref{eq:variational2}.
\end{proof}
A simple consequence of \eqref{eq:variational1}
and \eqref{eq:variational2} is the joint continuity of the map
$(\lambda,u_1,u_2)\mapsto \Wmetricl\lambda(u_1,u_2)$ in 
$(0,+\infty)\times \Xsp\times \Xsp.$
\begin{remark}
    In view of  the above discussion, we can also 
define $\Wmetricl\lambda$ using a  strict inequality
in \eqref{e:Wmetric-def}:
    \begin{equation}
    \label{e:Wmetric-def2}
        \Wmetricl\lambda(\xv_1,\xv_2)=
        \lambda \inf\Big\{r\ge0:\Wcost r{\xv_1}{\xv_2}< \lambda r\Big\}.
    \end{equation}
\end{remark}
\begin{example}
\label{ex:Wmetric-explicit}
It is interesting to compute $\Wmetricl\lambda$ in the case $\Wcostname$ is induced by 
a metric $d$ on $\Xsp$ 
as in Example \ref{ex:psi-metric}.
We have
\[
    \Wmetricl\lambda(\xv_1,\xv_2)=
       \lambda \inf\Big\{ r\ge0: \,  \Psidens (\tfrac1{r} d(\xv_1,\xv_2) )   \leq 
        \lambda\Big\}\,.
        \] 
        Therefore, if $\Psidens$ is invertible, we conclude that 
        \begin{equation}
        \label{explicit-Wmetric}
    \Wmetricl\lambda(\xv_1,\xv_2)     =  \frac{\lambda}{\Psidens^{-1}(\lambda)} d(\xv_1,\xv_2)\,.
        \end{equation}
\end{example}
\begin{theorem}
    \label{thm:metrizability}
   If $(\Xsp,\Wcostname)$ is an action space,
   then for all $\lambda>0$ the function 
   $\Wmetricl\lambda: \Xsp{\times}\Xsp \to [0,+\infty)$ defined by \eqref{e:Wmetric-def2}
   is a metric in $\Xsp$
    which
    satisfies
    \begin{equation}
        \label{e:useful-estimate}
        \Wmetricl\lambda(\xv_1,\xv_2)
        \le \lambda\tau\lor \Wcost\tau{\xv_1}{\xv_2}\le 
        \lambda \tau+\Wcost\tau{\xv_1}{\xv_2}
        \quad\text{for every }
        \tau>0,\ \xv_1,\xv_2\in \Xsp,
    \end{equation}
    and
    \begin{equation}
        \label{e:usefule-estimate2}
        \Wmetricl\lambda(\xv_1,\xv_2)<\lambda r\quad\Rightarrow\quad
        \Wcost{r}{\xv_1}{\xv_2}< \lambda r.
    \end{equation}
    In particular
    $\Wmetricl\lambda$ metrizes the uniform structure
    $\mathfrak U$ 
    (and a fortiori the topology $\Wtop$):
    setting
    \begin{equation}
        \label{eq:d-entourages}
        \mathrm B_\lambda(r):=
        \Big\{(\xv_1,\xv_2)\in \Xsp\times\Xsp:
        \Wmetricl\lambda (\xv_1,\xv_2)<\lambda r\Big\}
    \end{equation}
    we have:
    \begin{enumerate}
        \item for every $\tau,c>0$
        choosing $r\in (0,\tau\land (c/\lambda)]$
        we have 
        $\mathrm B_\lambda(r)
        \subset \VNei\tau c$;
        \item for every $r>0$
        choosing $\tau,c\in (0,r]$
        we have
        $\VNei\tau c\subset 
        \mathrm B_\lambda(r).$
    \end{enumerate}
\end{theorem}
\begin{proof}
    \eqref{e:useful-estimate} is an 
    trivial consequence of \eqref{eq:variational1}. In particular, 
    the above estimate shows 
    that $\Wmetricl\lambda$ takes
    finite values.
    Similarly, 
    \eqref{e:usefule-estimate2} follows immediately from \eqref{eq:variational2}.

    Properties (1) and (2) immediately follow
    by \eqref{e:usefule-estimate2}
    and \eqref{e:useful-estimate}
    respectively.
    
    Let us now check that $\Wmetric$ is a metric.
    Clearly $\Wmetric$ is symmetric and satisfies $\Wmetric(\xv,\xv)=0$.
   Property \eqref{e:usefule-estimate2}
    also shows that $\Wmetric(\xv_1,\xv_2)=0$
    implies $\xv_1=\xv_2.$

    Concerning the triangle inequality,
    let
    $r_1:=\Wmetricl\lambda(\xv_1,\xv_2)$
    and $r_2:=\Wmetricl\lambda(\xv_2,\xv_3)$,
    so that 
    $\Wcostminus{r_1}{\xv_1}{\xv_2}\le \lambda r_1$,
    $\Wcostminus{r_2}{\xv_2}{\xv_3}\le \lambda r_2$ by \eqref{e:upperd}.
    Recalling Proposition 
    \ref{prop:maybe-useful} we get
    $\Wcostminus{r_1+r_2}{\xv_1}{\xv_3}
    \le\lambda r_1+\lambda r_2$
    and therefore $\Wmetricl\lambda(\xv_1,\xv_3)\le 
    r_1+r_2.$
\end{proof}
 It is interesting to notice that 
passing to the limit as $\tau\downarrow0$ in 
\eqref{e:useful-estimate}
we get for all $\lambda>0$ and $u_1,\, u_2 \in \Xsp$
\begin{equation}
\label{e:boh}
    \Wmetricl\lambda({\xv_1},{\xv_2})\le  
    \lambda {\Wcostsupname}(u_0,u_1),
\end{equation}
 (cf.\ \eqref{a-sup} for the definition of
$\Wcostsupname$).

Thanks to the previous Theorem
we can use the metric $\Wmetric$
(or, equivalently, any of the equivalent
metrics $\Wmetricl\lambda$)
to characterize \emph{completeness}. 
\begin{corollary}
    The action space  
     $(\Xsp,\Wcostname)$
is \emph{complete}
if and only if 
$(\Xsp,\Wmetric)$
is a complete metric space.
\end{corollary}
\begin{example}[Completeness for action costs induced by a metric]
\label{ex:completeness-1}
\upshape
Let $\Wcostname$ be induced by a metric $d$  as in Example \ref{ex:psi-metric}, with an invertible $\uppsi$: then, by virtue of 
\eqref{explicit-Wmetric} it is straightforward to see that $(\Xsp,\Wmetric)$ is complete if and only if $(\Xsp,d)$ is complete.
\end{example}
\begin{example}[Completeness for linear combinations  of actions]
\label{ex:completeness-2}
Let $\Xsp$ be endowed with finitely many costs $\Wcostname_i$, $i \in I$.
Recall (cf.\ Example \ref{ex:2.8}) that,
for any family $(\teta_i)_{i\in I}$ of  positive coefficients,
$\Wcostname = \sum_i \theta_i \Wcostname_i$ is again an action cost on $\Xsp$. It is immediate to check that
the action space  
     $(\Xsp,\Wcostname)$ is complete if and only if  the spaces  
     $(\Xsp,\Wcostname_i)$
are complete for all $i\in I$.
\par
In particular, in view of Example \ref{ex:completeness-1},  if $\Xsp$ is endowed with two metrics $d_1$ and $d_2$, and complete w.r.t.\ both, 
then the action cost
$
\Wcost \tau uw = \tau \psidens{1}\left( \frac{\metr{1}uw}\tau \right)+ \tau \psidens{2}\left( \frac{\metr{2}uw}\tau \right) $
gives rise to a complete space.
\end{example}
In fact, in view of Theorem \ref{thm:metrizability}
 all the topological 
notions such  as convergence of sequences 
or continuity of functions
can be given in terms of $\Wmetric$. 
We collect the obvious definitions in
the following statements, 
\begin{corollary}
A sequence $(\xv_n)_n\subset \Xsp$ converges to some $\xv$ in the $\Wtop$-topology, and write
\[
\xv_n \toW \xv \text{ as $n \to \infty$,} \quad \text{or} \quad 
 \Wcostname\text{-}\lim_{n\to\infty} \xv_n = \xv, 
\]
if 
\[
\forall\, \tau>0, \, c >0 \quad \exists\, \bar n \in \N: \quad \forall\, n 
\geq \bar n\, \qquad \xv_n \in \Nei{\xv}{\tau}{c}
\,,
\]
or, equivalently, if
\begin{equation}
    \label{e:convergence}
    \lim_{n\to\infty}\Wcost\tau{\xv_n}\xv=0
    \quad\text{for every }\tau>0.
\end{equation}
 A sequence $(\xv_n)_n\subset \Xsp$  is  a \emph{Cauchy sequence}  
 in the $\mathfrak U$-uniform structure 
if it enjoys the following property:
\begin{equation}
\label{Cauchy-rewritten}
\forall\, \tau>0, \, c >0 \quad \exists\, \bar n \in \N:\quad 
\forall\, n,\, m \geq \bar n\,: \quad  (\xv_n,\xv_m) \in \TuNei{\tau}{c};
\end{equation}
equivalently, if
\begin{equation}
    \label{eq:Cauchy2}
    \lim_{m,n\to\infty}\Wcost{\tau}{\xv_m}{\xv_n}=0
    \quad\text{for every }\tau>0.
\end{equation}
\end{corollary}
%
%
 We can now easily discuss
the continuity of $\Wcostname$
and the lower-upper semicontinuity
of $\Wcostname_\pm$
 (recall \eqref{eq:plusminus} for the definition of $\Wcostname_{\pm}$). 
\begin{proposition}
    \label{prop:A-continuity}
    Let $(u_n)_n,(v_n)_n$
    be two sequences in $\Xsp$ 
    converging to $u,v$ in the $\Wtop$-topology
    and let $(\tau_n)_n$
    be a sequence in $(0,+\infty)$
    converging to $\tau>0$.
    We have
    \begin{equation}
        \label{eq:limits}
        \Wcostminus\tau uv
        \le \liminf_{n\to\infty}
        \Wcostminus {\tau_n}{u_n}{v_n}
        \le 
        \limsup_{n\to\infty}
        \Wcostplus {\tau_n}{u_n}{v_n}
        \le \Wcostplus\tau uv.
    \end{equation}
    In particular, 
    $\Wcostname_-$ (resp.~$\Wcostname_+$)
    is lower (resp.~upper) semicontinuous
    in $(0,+\infty)\times \Xsp\times \Xsp$
    with respect to the canonical product topology;
    if $\Wcostname$
    satisfies  property \eqref{e:continuity1},
    then it is continuous.
\end{proposition}
\begin{proof}
    Let us select $\tau'<\tau$ 
    and $\eps<\frac12(\tau-\tau').$
    For $n$ sufficiently large we can
    assume that $\tau_n>\tau'+2\eps$
    so that 
    the concatenation property 
    and the monotonicity of $\Wcostname_+$ yield
    \begin{align*}
        \Wcostplus {\tau_n}{u_n}{v_n}
        \nc
        &\le \Wcost {\tau'+2\eps}{u_n}{v_n}
        \le 
        \Wcost {\eps}{u_n}{u}+
        \Wcost {\tau'}{u}{v}+
        \Wcost {\eps}{v}{v_n}
    \end{align*}
    Taking the $\limsup$ of the left hand side
    as $n\to\infty$ and using the 
    convergence property 
    \eqref{e:convergence} we get
    \begin{equation*}
        \limsup_{n\to\infty}
        \Wcostplus {\tau}{u_n}{v_n}
        \le \Wcost {\tau'}{u}{v}
    \end{equation*}
    Since $\tau'<\tau$ is arbitrary
    we obtain 
    $\limsup_{n\to\infty}
        \Wcostplus {\tau_n}{u_n}{v_n}
        \le \Wcostplus\tau uv$.

    Similarly, selecting $\tau''>\tau$
    and $\eps>0$ such that $\tau''-2\eps>\tau$
    we get for $n$ sufficiently large
    \begin{align*}
     \Wcostminus {\tau''}{u}{v}\nc
        \le \Wcost {\tau_n+2\eps}{u}{v}
        \le \Wcost {\eps}{u}{u_n} 
        +\Wcost {\tau_n}{u_n}{v_n}+
        \Wcost {\eps}{v_n}{v}.
    \end{align*}
    Taking the $\liminf$ of the right-hand side
    and then the supremum with respect to $\tau''$ we obtain 
    $\liminf_{n\to\infty}
        \Wcostminus {\tau_n}{u_n}{v_n}
        \ge \Wcostminus\tau uv$.
\end{proof}

 It is interesting
to consider the 
interaction of 
 action costs 
with an auxiliary topology $\sigma$ on
$\Xsp$. 
This is useful 
when metric costs are involved in a Minimizing Movement scheme
and, typically, the driving energy functional has compact sublevels w.r.t.~$\sigma$. 
Our next result shows that
(sequential) lower semicontinuity of $\Wcostname_-$ 
w.r.t.\ $\sigma$ leads to (sequential) lower semicontinuity of $\Wmetric$.  


\begin{proposition}
Let $(\Xsp,\Wcostname)$ be an action space 
and let $\sigma$ be a Hausdorff topology 
on $\Xsp$ for which 
$\Wcostname_-$ is sequentially lower semicontinuous, i.e.~
for  all  sequences $(\xv_i^n)_{n\in \N}$,
 $\sigma$-converging to 
$\xv_i$, $i=0,1$, we have
\begin{equation}\label{lower-semicont}
\liminf_{n \to \pinfty} \Wcostname_-(\tau,\xv_0^n, \xv_1^n) \geq \Wcostname_-(\tau,\xv_0,\xv_1)
\quad\text{for every }\tau>0.
 	\end{equation}
Then also $\Wmetric$ is 
$\sigma$-sequentially lower semicontinuous,
i.e.\
for  all  sequences $(\xv_i^n)_{n\in \N}$,
 $\sigma$-converging to 
$\xv_i$, $i=0,1$, we have
\begin{equation}\label{d-lower-semicont}
\liminf_{n \to \pinfty} \Wmetric(\xv_0^n, \xv_1^n) \geq \Wmetric(\xv_0,\xv_1).
 	\end{equation}
In particular,
every 
$\sigma$-sequentially compact subset $K\subset \Xsp$ 
is $\Wcostname$-complete and
every $\Wtop$-convergent sequence
is also $\sigma$-convergent.
\end{proposition} 
\begin{proof}
It is sufficient to prove that 
\eqref{lower-semicont} implies
\eqref{d-lower-semicont}.
For that, let us show
that if
$\Wmetric(\xv_0^n,\xv_1^n)\le r$
definitely, then also
$\Wmetric(\xv_0,\xv_1)\le r.$  From characterization
\eqref{e:upperd} it follows
 that 
$\Wcostminus r{\xv_0^n}{\xv_1^n}\le r$.
Passing to the limit as $n\to\infty$ and 
using 
\eqref{lower-semicont}
we deduce that 
$\Wcostminus r{\xv_0}{\xv_1}\le r$,
so that 
$\Wmetric(\xv_0,\xv_1)\le r$
as well.
\end{proof}
\begin{remark}
    It is worth noticing that sequential lower semicontinuity of $\Wcostname$
    w.r.t.~$\sigma$ implies
    the same property for $\Wcostname_-$.
\end{remark}
    
 \nc
\section{Curves with finite action}
\label{s:4}
In this section we assume
that $(\Xsp,\Wcostname)$ is an action space.
\begin{definition}[Action of a curve]
    Let    
    $u: [a,b]\to \Xsp$.
    The
    $\Wcostname$-action of $u$ is defined by
\begin{equation}
\label{VarW-variation}
\VarW u ab: = \sup \left \{ \sum_{j=1}^M  
\Wcost{t^j - t^{j-1}}{u(t^{j-1})}{u(t^j)} \, : \ (t^j)_{j=0}^M \in \mathscr{P}_f([a,b])   \right\}
\end{equation}
where $\mathscr{P}_f([a,b])$  is the set of all finite partitions of the interval $[a,b]$. 
\end{definition}
This definition is clearly reminiscent of   that of the total variation 
$\Var_{\mathsf d}$  induced by 
a metric $\mathsf d$, i.e.\
\begin{equation}
\label{VarW-metric}
  \mathrm{Var}_{\mathsf d}(u;[a,b]): = \sup \left \{ \sum_{j=1}^M  
\mathsf d(u(t^{j-1}),u(t^j)) \, : \ (t^j)_{j=0}^M \in \mathscr{P}_f([a,b])   \right\}\,.
\end{equation}
Now, thanks to  \eqref{e:useful-estimate}  
we have
a simple estimate
of the   $\mathrm{Var}_{\Wmetricl\lambda}$-variation 
in terms of $\VarWname u$.
\begin{lemma}
    For every 
    curve $u:[a,b]\to\Xsp$
    \begin{equation}
    \label{estimate-variations}
    \mathrm{Var}_{\Wmetricl\lambda}(u;[a,b])
    \le \lambda\big(b-a)+
    \VarW u ab.
\end{equation}
In particular, if $u$
    has finite action 
    $\VarW u ab<\infty$
    then it also has
    finite $\Wmetricl\lambda$-variation.
\end{lemma}
Hence, 
assuming completeness of $(\Xsp,\Wcostname)$  we show that
curves of finite $\Wcostname$-action
are \emph{regulated}, and  indeed  have $\BV$-like properties.  
%
%
%
%
%
\begin{proposition}
\label{prop:regulated}
If the action space
$(\Xsp,\Wcostname)$ is complete 
then every curve 
$u:[a,b]\to \Xsp$ 
with finite action $\VarW u ab<+\infty$ satisfies
\begin{subequations}
\label{regulated}
\begin{align}
&
\forall\, t \in (a,b] \quad \exists\, \llim ut: = 
\Wtop\text{-}\lim_{s\uparrow t} u(s),
\\
&
\forall\, t \in [a,b) \quad \exists\, \rlim ut: = 
\Wtop\text{-}\lim_{s\downarrow t} u(s)
\end{align}
\end{subequations}
(we also adopt the convention
$\llim ua := u(a)$ and $\rlim u b: =u(b)$). 
Furthermore, the pointwise  jump set 
\[
\mathrm{J}_u: =\mathrm{J}_u^+ \cup \mathrm{J}_u^-
\qquad \text{with }
\begin{cases}
\mathrm{J}_u^-: =  \{ t \in [a,b]\, : \ \llim u t \neq u(t) \},
\\
 \mathrm{J}_u^-: =  \{ t \in [a,b]\, : \ u(t) \neq 
\rlim u t\}
 \end{cases}
\]
consists of at most countably many points, and the function $\Tvar Au : [a,b] \to [0,+\infty)$, $\Tvar Au(t): = \VarW u at$, has bounded variation.
\end{proposition} 
\begin{proof}
From  $\VarW u ab<+\infty$ we infer that $  \mathrm{Var}_{\Wmetric}(u;[a,b])<+\infty$ via \eqref{estimate-variations}. Since  the metric space
$(\Xsp,\Wmetric)$ is complete, any $u\in \mathrm{BV}_{\Wmetric}([a,b];\Xsp)$ admits left- and right-limits w.r.t.\ the topology induced by $\Wmetric$, whence
 \eqref{regulated}. Furthermore, $\mathrm{J}_u $ coincides with the (analogously defined)  jump set of the bounded variation function
 $\Tvar Vu :[a,b]\to [0,+\infty)$, $\Tvar Vu(t): =  \mathrm{Var}_{\Wmetric}(u;[a,t]) $, and thus $u$ has countably many jump points.
\end{proof}
Our next result shows that, if the cost $\Wcostname$ has local superlinear growth 
according to \eqref{e:superlinear}
for any curve $u$ with finite action 
$u$ is $\Wtop$-continuous at any point.
\begin{proposition}
Suppose that  $(\Xsp,\Wcostname)$ is complete and 
$\Wcostname$  has superlinear local growth.
Then, every curve $u:[a,b]\to \Xsp$ with finite action  $\VarW u ab<+\infty$
is continuous.
\end{proposition}
\begin{proof}
We fix $t\in (a,b)$ 
(similar arguments can be carried out for $t=a$ and $t=b$), $\tau>0$, 
and a strictly positive 
vanishing sequence $(\eta_n)_n$
with $\eta_n \down 0$. For $n$ sufficiently large we have
\begin{align*}
\Wcost {\tau} {u(t{-}\eta_n)}{u(t)} + \Wcost {\tau} {u(t)}{u(t{+}\eta_n)} 
&\le 
  \Wcost {\eta_n} {u(t{-}\eta_n)}{u(t)} + \Wcost {\eta_n} {u(t)}{u(t{+}\eta_n)} \\&\le 
  \VarW uab<+\infty.
\end{align*}
Taking the limit inferior as $n\to\infty$ 
we get
\begin{equation*}
    \Wcostminus {\tau} {u_-(t)}{u(t)} + \Wcostminus {\tau} {u(t)}{u_+(t)} 
    \le \VarW uab.
\end{equation*}
We can now pass to the limit
 in the above estimate along
a vanishing sequence $(\tau_n)_n $ 
such that 
$\Wcostminus {\tau_n} {u_-(t)}{u(t)}=
\Wcost {\tau_n} {u_-(t)}{u(t)}$
and 
$\Wcostminus {\tau_n} {u(t)}{u_+(t)} 
=\Wcost {\tau_n} {u(t)}{u_+(t)} $
obtaining that 
$\Wcostsup{u_-(t)}{u(t)}+
\Wcostsup{u(t)}{u_+(t)}<\infty.
$
The local superlinearity of $\Wcostname$
then yields $u_-(t)=u(t)=u_+(t).$
\end{proof}

\section{Absolute continuity}
\label{s:5}
 We now introduce a notion of absolute continuity for curves with values in an action space. 
\begin{definition}[Absolutely continuous curves]
\label{def:AC}
We say that a curve $u:[a,b]\to \Xsp$ is $\DVTn$-absolutely continuous  if there exists $g\in L^1(a,b) $ such that 
\begin{equation}
\label{def-AC-curve}
\DVT{t{-}s}{u(s)}{u(t)} \leq \int_s^t g(r) \dd r \qquad \text{for all } a \leq s \leq t \leq b.
\end{equation}
\end{definition}
\begin{theorem}[Action density]
\label{l:surr-metr-der}
Let $u\in\AC_{\DVTn}([a,b];\Xsp)$. Then, 
the limit
\begin{equation}
\label{surr-mder}
\lim_{\sigma \down t} \frac{\DVT{\sigma{-}t}{u(t)}{u(\sigma)}}{\sigma-t} = \lim_{\sigma \up t} \frac{\DVT{t{-}\sigma}{u(\sigma)}{u(t)}}{t-\sigma}  =:
\smd ut \qquad \text{exists at almost all } t \in (a,b)
\end{equation}
and it fulfills
\begin{subequations}
\label{minimality}
\begin{align}
&
\label{minimality-a}
\DVT{t{-}s}{u(s)}{u(t)} \leq \int_s^t \smd ur \dd r \qquad \text{for all } a \leq s \leq t \leq b,
\\
&
\label{minimality-b}
\smd ut \leq g(t) \qquad \foraa\, t \in (a,b)
\end{align}
\end{subequations}
for every $g\in L^1(a,b)$ such that \eqref{def-AC-curve} holds. Therefore, $\smdn u \in L^1(a,b)$.  We shall refer to $\smdn u$ as
 the  action density for $u$. 
\end{theorem}
\begin{proof}
The proof closely follows the argument for \cite[Prop.\ 3.2]{RMS08}.  For every fixed $s\in [a,b)$, we introduce the function 
\begin{equation}
\label{auxiliary-cost}
\ell_s : (s,b] \to [0,\infty) \qquad t \mapsto \ell_s(t):= \DVT{t-s}{u(s)}{u(t)}
\end{equation}
and observe that, by the triangle inequality \eqref{e:psi3}, there holds
\[
\left( \ell_s(t_2){-}\ell_s(t_1) \right)^+ \leq \DVT{t_2{-}t_1}{u(t_1)}{u(t_2)} \leq   \int_{t_1}^{t_2} g(r) \dd r  \qquad \text{for all }  s<t_1 <t_2\leq b\,.
\]
Therefore, the map $t\mapsto \ell_s(t)-\int_s^t g(r) \dd r $
 is non-increasing on $(s,b],$ and thus it is a.e.\ differentiable, with 
 \[
 ( \ell_s'(t))^+ \leq  \mathfrak{a}_- (t) : = \liminf_{\sigma \down t}  \frac{\DVT{\sigma{-}t}{u(t)}{u(\sigma)}}{\sigma-t}  \qquad \foraa\, t \in (s,b),
 \]
 Observe that $ \mathfrak{a}_- $ is itself a measurable function, fulfilling 
 $0 \leq   \mathfrak{a}_- (t)  \leq g(t)$ for almost all $t \in (a,b)$ (with the second inequality due to \eqref{def-AC-curve}).
  Thus, $ \mathfrak{a}_-  \in L^1(a,b)$. With the very same argument as in the proof  of  \cite[Prop.\ 3.2]{RMS08} we deduce that 
  \begin{equation}
  \label{W_+-adm}
   \DVT{t-s}{u(s)}{u(t)}  = 
\ell_s(t)  \leq \int_s^t  ( \ell_s'(r))^+ \dd r \leq \int_s^t   \mathfrak{a}_- (r) \dd r \qquad \text{for all } [s,t]\subset (a,b]\,.
  \end{equation}
  Finally, we consider the function 
  \[
   \mathfrak{a}_+  : (a,b) \to [0,\infty), \qquad t \mapsto 
  \mathfrak{a}_+ (t) := \limsup_{\sigma \down t}  \frac{\DVT{\sigma{-}t}{u(t)}{u(\sigma)}}{\sigma-t} \,
  \]
and observe that 
\begin{equation}
\label{minimality-W}
 \mathfrak{a}_+ (t) \leq g(t) \quad \text{for almost all  $t\in (a,b)$, for any function $g$ for which  \eqref{def-AC-curve} holds.}
\end{equation}  In view of \eqref{W_+-adm}, we may choose $g=  \mathfrak{a}_- $ and we thus conclude that 
\[
 \limsup_{\sigma \down t}  \frac{\DVT{\sigma{-}t}{u(t)}{u(\sigma)}}{\sigma-t}  = \mathfrak{a}_+(t) \leq  \liminf_{\sigma \down t}  \frac{\DVT{\sigma{-}t}{u(t)}{u(\sigma)}}{\sigma-t}   =  \mathfrak{a}_-(t) \qquad \foraa\, t \in (a,b)\,.
\]
Therefore, we ultimately conclude that $\lim_{\sigma \down t} \frac{\DVT{\sigma{-}t}{u(t)}{u(\sigma)}}{\sigma-t}$ exists at almost all $t\in (a,b)$, and, also in view of \eqref{W_+-adm}, 
 that it fulfills the minimality properties
\eqref{minimality}.
The proof of the assert for  $\lim_{\sigma \up t} \frac{\DVT{t{-}\sigma}{u(\sigma)}{u(t)}}{t-\sigma}$ can be trivially adapted from the analogous argument in \cite{RMS08}, to which we refer the reader for all details.
\end{proof}

\begin{theorem}
\label{l:consistency}
For every  $u \in \AC_{\DVTn}([a,b];\Xsp)$ we have 
$
\VarW uab<\infty$ and 
\begin{equation}
\label{consistency}
\VarW uab= \int_a^b \smd u t \dd t \,.
\end{equation}
\end{theorem}
\begin{proof}
Let us fix an arbitrary partition $  (t^j)_{j=0}^M \in \mathscr{P}_f([a,b]) $ and observe that, for every $j=1,\ldots, M$, 
\[
\DVT{t^j - t^{j-1}}{u(t^{j-1})}{u(t^j)} \leq \int_{t^{j-1}}^{t^j} \smd ur \dd r\,.
\]
Therefore, 
\[
\sum_{j=1}^M \DVT{t^j - t^{j-1}}{u(t^{j-1})}{u(t^j)} \leq \int_a^b \smd ur \dd r
\]
and, by the arbitrariness of $  (t^j)_{j=0}^M $, we conclude that 
\begin{equation}
\label{one-sided}
\VarW uab\leq  \int_a^b \smd u t \dd t.
\end{equation} 
\par
\par
In order to show the converse of  inequality \eqref{one-sided},   
let  now $(\mathcal{P}_k)_{k \in \N} \subset  \mathscr{P}_f([a,b]) $ be a sequence of uniform partitions of size 
$\tau(k):=(b-a)/k$, $\mathcal{P}_k = (t_k^j)_{j=0}^{k}$ where
$t_k^j=a+j\tau(k)$. \nc
 Let us introduce
 the piecewise constant functions
associated with $\mathcal{P}_k$
\begin{equation}
\label{interpolants-nodes}
\begin{aligned}
& \tpart {k} :[a,b]\to [a,b], 
&& \tpart {k}(a): = a, && \tpart {k}(t): = t^j_k  \ \text{ if } t \in (t_k^{j-1}, t_k^j],
\\
 & \utpart {k} : [a,b]\to [a,b], && \utpart {k}(b): = b, && \tpart {k}(t): = t^{j-1}_k \ \text{ if } t \in [t_k^{j-1}, t_k^j),
\end{aligned}
\end{equation}
and hence associate with $u$ the functions
\begin{equation}
\label{discretized-u}
\begin{aligned}
& \fpart u {k} :[a,b]\to \Xsp, &&  \fpart u {k}(t): = u(\tpart {k} (t)),
\\
& \ufpart u {k}:[a,b]\to \Xsp, &&  \ufpart u {k}(t): = u(\utpart {k} (t))\,.
\end{aligned}
\end{equation}
 We also introduce  the functions
\[
\Upsilon_{k} (t): = \frac{\displaystyle \Wcost{ t{-}\utpart {k} (t)}{\ufpart u {k} {(t)}}{u(t)} +
\Wcost{ \tpart {k} (t){-} t}{u(t)}{\fpart u {k} {(t)} }}{\displaystyle \tau(k)
}, 
\qquad t \in [a,b]\,,
\]
and observe that for every $t\not\in\mathcal P_k$
\begin{align*}
    \Upsilon_{k} (t) = 
    \alpha_k(t)
    \frac{\displaystyle \Wcost{ t{-}\utpart {k} (t)}{\ufpart u {k} {(t)}}{u(t)}}
    {t-\utpart k(t)}
    +
    \beta_k(t)
\frac{\Wcost{ \tpart {k} (t){-} t}{u(t)}{\fpart u {k} {(t)} }}{\displaystyle \tau(k)}
\end{align*}
where
\begin{equation*}
    \alpha_k(t):=\frac {t-\utpart k(t)}{\tau(k)},\quad
    \beta_k(t):=\frac{\tpart k(t)-t}{\tau(k)},\quad
    \alpha_k(t),\beta_k(t)\ge 0,\quad
    \alpha_k(t)+\beta_k(t)=1.
\end{equation*}
We have that \nc
 \[
\lim_{k \to \infty}  \Upsilon_{k} (t)
= \smd u t \qquad \foraa\, t \in (a,b).
\]
%
%
%
%
%
%
Hence,  by the Fatou Lemma we find
\[
\begin{aligned}
 \int_a^b  \smd u t  \dd t  \leq 
  \liminf_{k\to \infty}\int_a^b  \Upsilon_{k} (t)
 \dd t   .
 \end{aligned}
\]
On the other hand we observe that 
\begin{align*}
    \int_a^b  \Upsilon_{k} (t)
 \,\dd t &=
 \sum_{j=0}^{k-1} 
 \int_{t_k^{j}}^{t_k^{j+1}}
 \Upsilon_{k} (t)
 \,\dd t
 \\&=
 \sum_{j=0}^{k-1} 
 \int_{0}^{\tau(k)}
 \Upsilon_{k} (a+j\tau(k)+s)
 \,\dd s
 \\&=
 \int_{0}^{\tau(k)}
 \bigg(\sum_{j=0}^{k-1} 
 \Upsilon_{k} (a+j\tau(k)+s)
 \bigg)\,\dd s
 \\&=
 \frac1{\tau(k)}
 \int_0^{\tau(k)}
 \bigg(
 \sum_{j=0}^{k-1}
 \Wcost{ s}{u(a+j\tau(k))}{u(a+j\tau(k)+s)}
 \\&\qquad\qquad\qquad\qquad+
\Wcost{ \tau(k)-s}{u(a+j\tau(k)+s)}{u(a+(j+1)\tau(k))}
\bigg)\,\dd s
\\&\le 
\frac1{\tau(k)}
 \int_0^{\tau(k)}
 \VarW uab\,\dd s
 =\VarW uab,
\end{align*}
which finishes the proof. \end{proof}
\begin{corollary}
    If $u\in \AC_{\Wcostname}([a,b];\Xsp)$
    then 
    $u\in \AC([a,b];(\Xsp,\Wmetricl\lambda))$
    for every $\lambda>0$
    and
    \begin{equation}
        \label{eq:metric-estimate}
        |u'|_{\Wmetricl\lambda}\le 
        \lambda \lor \mathfrak a[u'].
    \end{equation}
\end{corollary}
\begin{proof}
    Combining 
    \eqref{e:useful-estimate}
    with \eqref{def-AC-curve}
    one immediately sees that 
    $u$ is absolutely continuous with respect to $\Wmetricl\lambda.$
    We can then use \eqref{e:useful-estimate}
    and \eqref{surr-mder}
    to deduce \eqref{eq:metric-estimate}.
\end{proof}

\section{Sufficient conditions for absolute continuity}
\label{s:6}
In this section  we address 
 the converse of 
Theorem \ref{l:consistency} hold, namely
we examine the validity of the implication
\begin{equation}
\label{converse-AC}
\VarW uab<\infty \ \Rightarrow  u \in \AC_{\DVTn}([a,b];\Xsp).
\end{equation}
It is immediate to realize that, in the case of the action integral costs from Section  \ref{ss:Finsler},  the existence of an action minimizing curve 
(cf.\ Proposition \ref{prop:action-int-props}) guarantees the validity of \eqref{converse-AC}. 
\par
In a different spirit,
we  propose the following property  of  $\Wcostname$, cf.\ \eqref{eq:uniform-superlinearity} below, as a sufficient condition for     \eqref{converse-AC}. 
\begin{definition}[Uniform superlinearity]
\label{ass:ETI}
We  say that  the  action cost $\Wcostname$ on $\Xsp$
is \emph{uniformly superlinear}
if there exists   an 
 action cost  $\Gacostname$ on $\Xsp$,
a convex superlinear function
$\Psidens:[0,+\infty)\to[0,+\infty)$,
and a  constant $\lambda\ge 1$
such that 
\begin{equation}
    \label{eq:uniform-superlinearity}
    \lambda^{-1}\tau \Psidens\Big(\tau^{-1}\Gacost\tau uv\Big)\le \Wcost\tau uv\le 
    \lambda\tau \Psidens\Big(\tau^{-1}\Gacost\tau uv\Big)
\end{equation}
for every $\tau>0$, $u,v\in \Xsp.$
%
\end{definition}
\begin{remark}
\label{rmk:relation-to-1-metric}
\upshape
Definition \ref{ass:ETI} somehow mirrors the 
convex construction of Example 
\ref{ex:psi-construction}
and it 
is clearly satisfied whenever $\Psidens$ is superlinear.   In particular, action costs induced by a metric as in Example \ref{ex:psi-metric} comply with \eqref{eq:uniform-superlinearity}. 
\end{remark}
\par
We have the following result. 
\begin{theorem}
\label{thm:AC1}
If $\Wcostname$ is a uniformly superlinear
 action cost  on $\Xsp$ then 
for every $u:[a,b]\to \Xsp$ such that  $ \VarW uab<\infty $
we have $ u \in \AC_{\DVTn}([a,b];\Xsp)$.
\end{theorem}
\begin{proof}
Since $\Psidens$ is superlinear,
there exists a constant $\beta\ge0 $
such that 
$\Psidens(r)\ge r-\beta$
for all $r\in [0,+\infty)$, so that 
$\tau\Psidens(\tau^{-1}r)\ge r-\tau\beta$ and 
\begin{equation}
\label{Ga&W}
 \Gacost \tau uv 
 \le 
 \lambda \Wcost \tau uv +\tau\beta
 \qquad \text{for all } (\tau,u,v)\in (0,+\infty){\times}\Xsp{\times}\Xsp.
\end{equation}
From the above inequality it follows that 
any $u:[a,b]\to \Xsp$ 
satisfies
\[
\mathbb B(u;[a,b])\le 
\lambda \VarW uab+\beta (b-a)
\]
where $\mathbb B$ denotes the action functional induced
by $\Gacostname$.
In particular
\[ \VarW uab<\infty \quad\Rightarrow\quad
\mathbb B(u;[a,b])<\infty.\]
Therefore, the function  $\Tvar Bu :[a,b]\to [0,+\infty)$, $\Tvar Bu(t): =  \mathbb B
(u;[a,t]) $ has bounded variation. Let $\nu_{u}$ be its distributional derivative. 
  Let now $(\mathcal{P}_k)_{k \in \N} \subset  \mathscr{P}_f([a,b]) $ be the sequence of uniform partitions
considered in the proof of Theorem
\ref{l:consistency}. 
 With the same notation as in \eqref{interpolants-nodes} and \eqref{discretized-u}, 
 we consider now the piecewise constant functions
\[
\Gamma_{k} (t): = \frac{\displaystyle \Gacost{ \tpart {\mathcal{P}_k} (t){-}\utpart {\mathcal{P}_k} (t)}{\ufpart u {\mathcal{P}_k} (t)}{\fpart u {\mathcal{P}_k} {(t)} }}{\displaystyle \tpart {\mathcal{P}_k} (t){-}\utpart {\mathcal{P}_k} (t)}, \qquad t \in [a,b]\,.
\]
Then, the measures $\nu_k : =  \Gamma_{k} \mathscr{L}^1$ (where $ \mathscr{L}^1$ denotes the Lebesgue measure on $[a,b]$) weakly$^*$ converge to $\nu_{u}$. 
In turn, we observe that for every $t\in (a,b)$
\[
\begin{aligned}
\Psidens \big(\Gamma_{k} (t)\big)  & \le 
 \lambda \, \frac1{\displaystyle \tpart {\mathcal{P}_k} (t){-}\utpart {\mathcal{P}_k} (t)} \Wcost{ \tpart {\mathcal{P}_k} (t){-}\utpart {\mathcal{P}_k} (t)}{\ufpart u {\mathcal{P}_k} (t)}{\fpart u {\mathcal{P}_k} {(t)} }\,,
 \end{aligned}
\]
so that 
\begin{equation}
\label{key-estimate}
\begin{aligned}
\sup_{k\in \N} \int_a^b \Psidens \big(\Gamma_{k} (t)\big) \dd t  & \leq
\lambda \sup_{k\in \N} \int_a^b  \frac1{\displaystyle \tpart {\mathcal{P}_k} (t){-}\utpart {\mathcal{P}_k} (t)} \Wcost{ \tpart {\mathcal{P}_k} (t){-}\utpart {\mathcal{P}_k} (t)}{\ufpart u {\mathcal{P}_k} (t)}{\fpart u {\mathcal{P}_k} {(t)} } \dd t
\\
 &  \leq  \lambda \, \VarW uab<+\infty\,.
 \end{aligned}
\end{equation}
Since the convex function $\Psidens $ has superlinear growth at infinity, 
by the well known De Vall\'ee-Poussin criterion
we conclude that the bounded sequence $(\Gamma_k)_k  \subset L^1(a,b)$  admits a (non-relabeled) subsequence weakly converging to some $\Gamma \in L^1(a,b)$, so that $\nu_u = \Gamma \mathcal{L}^1$. It then follows from \eqref{key-estimate}  that 
\[
\int_a^b \Psidens(\Gamma(t)) \dd t \leq \liminf_{k\to\infty} \int_a^b \Psidens (\Gamma_{k} (t)) \dd t  \leq  \lambda \VarW uab<+\infty\,.
\]
Now we have for all $a\leq s \leq t \leq b$
\[
\begin{aligned}
\Wcost{t{-}s}{u(s)}{u(t)}  
&\le \lambda (t{-}s) \Psidens \left(\tfrac1{t{-}s} \Gacost {t{-}s} {u(s)}{u(t)} \right)  \leq  
\lambda (t{-}s) \Psidens \left(\tfrac1{t{-}s} \nu_u ([s,t]) \right) 
\\ & =   \lambda (t{-}s) \Psidens \left(\frac1{t{-}s} \int_s^t \Gamma(r) \dd r  \right) 
\stackrel{(1)}{\leq} \lambda \int_s^t \Psidens(\Gamma(r))\dd r  
 \end{aligned}
\]
where {\footnotesize (1)} due to  the  Jensen inequality. Since $\Psidens{\circ}\Gamma \in  L^1(a,b)$, we conclude that $u \in \AC_{\DVTn}([a,b];\Xsp)$ and that, in fact,
 \begin{equation}
 \label{estimate-fGamma}
 \smd ut \leq  \lambda \Psidens(\Gamma(t)) \qquad \foraa\, t \in (a,b)
 \end{equation}
(cf.\ the minimality property \eqref{minimality-b}). This finishes the proof.
\end{proof}

\begin{remark}
\label{rmk:added-later}
\upshape
Revisiting the proof of Theorem \ref{thm:AC1} we observe that, a fortiori, 
$u \in \AC_{\Gacostname}([a,b];\Xsp)$, since $u \in \AC_{\Wcostname}([a,b];\Xsp)$ and $\Wcostname$ dominates $\Gacostname$, cf.\ \eqref{Ga&W}. 
Now, we immediately check that 
\begin{equation}
    \label{double-estimate}
\lambda^{-1} \Psidens \Big(\mathfrak{b}[u'](t)
\Big)\le  \smd ut
\le \lambda \Psidens \Big(\mathfrak{b}[u'](t)
\Big)\qquad \foraa\, t \in (a,b)\,.
\end{equation}
 Indeed, \ since for all $[s,t]\subset [a,b]$ we have 
\[
\begin{aligned}
\int_s^t \Gamma(r) \dd r  &  = \lim_{k\to\infty} \int_s^t  \Gamma_k(r) \dd r
\\
 & 
 \leq \lim_{k\to \infty}  \int_s^t \left(  \frac1 {\displaystyle \tpart {\mathcal{P}_k} (r){-}\utpart {\mathcal{P}_k} (r)} \int_{\utpart {\mathcal{P}_k} (r)}^{\tpart {\mathcal{P}_k} (r)}
\mathfrak{b}[u'](\omega)  \dd \omega  \right) \dd r  =  \int_s^t \mathfrak{b}[u'](r) \dd r
\end{aligned}
\]
we conclude that  
\[
\Gamma(t) \leq \mathfrak{b}[u'](t)  \qquad \foraa\, t \in (a,b)\,,
\]
and thus 
\[
\smd ut  \le \lambda  \Psidens( \mathfrak{b}[u'](t) )  \qquad \foraa\, t \in (a,b)\,.
\]
 The other inequality in \eqref{double-estimate} 
follows by a similar argument.
\end{remark}

\subsection*{Example $1$: Cost induced by a metric}
\upshape
If $\Wcostname$ has the structure $ 
 \Wcost \tau u v = \tau \Psidens (\tfrac1\tau d(u,v))$
  for a strictly increasing convex superlinear $\Psidens$, 
 we immediately infer from 
 Theorem \ref{thm:AC1} and Remark \ref{rmk:added-later} 
  the following
 \begin{corollary}
 \label{cor:one-metric}
  For every
 $ u: [a,b]\to \Xsp $ with $  \VarW uab<\infty$ we have 
   $u\in\AC_{\DVTn}([a,b];\Xsp) \subset   \AC_{d}([a,b];\Xsp)  $ and  there holds 
  \begin{equation}
  \label{formula-metric-derivative}
  \smd ut  =  
  \psidens{}( |\mathfrak{u}'|_{\metrname{}}(t) )\qquad \foraa\, t \in (a,b)\,.
  \end{equation}
  \end{corollary}
    \begin{proof}
  It suffices to observe that, 
     in the construction set up in the proof of Theorem \ref{thm:AC1}, we have in this case (cf.\ Remark \ref{rmk:relation-to-1-metric})
  $\Gacost \tau uv = d(u,v) $ for all $u,v\in \Xsp$, 
  $\lambda=1$  
   and $\Gamma(t) =  |\mathfrak{u}'|_{\metrname{}}(t) $ for almost all $t\in (a,b)$.
    \end{proof}   
\subsection*{Example $2$: cost induced by  two metrics}
We now address the case in which $\Xsp$ is endowed with two distances $\metrname{1}$ and $\metrname{2}$ satisfying 
 \begin{equation}
 \label{spazio-metric3} 
  \exists\, K>0\, \ \ \forall\, u, \, v \in \Xsp\: \quad \ \metr{1}uv
  \leq K \metr{2}uv\,.
  \end{equation} 
Since by the latter condition $d_2$ `dominates' $d_1$, in this case we obviously have  
\[
\AC_{\metrname{2}}([a,b];\Xsp) \subset \AC_{\metrname{1}}([a,b];\Xsp).
\]
 As in Section \ref{s:example}, we consider metric costs of the form
\begin{equation}
\label{structure-W-1+2-rewritten}
\Wcost \tau uw = \tau \psidens{1}\left( \frac{\metr{1}uw}\tau \right)+ \tau \psidens{2}\left( \frac{\metr{2}uw}\tau \right) \qquad \text{for all } (\tau,u,w) \in (0,+\infty)\times \Xsp \times \Xsp
\end{equation}
with 
$\psidens 1,\, 
\psidens 2 :  [0,+\infty) \to [0,+\infty)$
 convex, such that $0=\psidens i(0)<\psidens i(a)$  for all $a>0$ (cf.\ Example \ref{ex:psi-construction}). 
 We have the following result.
\begin{proposition}
\label{prop:AC1+2}
Assume \eqref{structure-W-1+2-rewritten} with $\psidens 2$ invertible.  Then, 
for every $u:[a,b]\to \Xsp$ such that
  $\VarW uab<\infty $
 we have  $ u \in \AC_{\DVTn}([a,b];\Xsp)$ and 
\begin{equation}
\label{identification-W-mder}
\smd ut  = \psidens{1}(\mder{1} ut )+  \psidens{2}(\mder{2} ut ) \qquad \foraa\, t \in (a,b). 
\end{equation}
\end{proposition}
\begin{proof}
From $ \VarW uab<\infty $ we now gather, in particular,  that $  \VarWd uab<\infty $,  where $  \VarWdname $ is  the action functional associated with $\Wcostname_2(\tau,u,v) = \tau\psidens 2 (\tfrac1\tau \metr 2 uv)$. 
%
Then, by Corollary \ref{cor:one-metric} we have 
$u \in \AC_{\metrname{2}}([a,b];\Xsp)$, and hence 
$u\in \AC_{\metrname{1}}([a,b];\Xsp)$. Then,
for almost all $t\in (a,b)$
 \[
 \begin{aligned}
\smd ut  =  \lim_{\sigma \down t} \frac{\DVT{\sigma{-}t}{u(t)}{u(\sigma)}}{\sigma-t} &  =   \lim_{\sigma \down t}  \psidens{1}\left(\frac{\metr{1}{u(t)}{u(\sigma)}}{\sigma-t} \right) 
+ \lim_{\sigma \down t}  \psidens{2}\left(\frac{\metr2{u(t)}{u(\sigma)}}{\sigma-t} \right) 
\\
&
= \psidens{1}( |\mathfrak{u}'|_{\metrname{1}}(t) )+ \psidens{2}( |\mathfrak{u}'|_{\metrname{2}}(t) )\,.
\end{aligned}
 \]
This finishes the proof.
\end{proof}

\section{A dynamic interpretation of 
 action costs}
\label{s:7}
\noindent
The goal of this section is to retrieve, for the    action cost  $\Wcostname$, a dynamic interpretation akin to 
\eqref{Finsler-cost-intro} and \eqref{dvt-intro} for  action integral costs
and Dynamical-Variational Transport
costs, respectively. 
We will obtain  this   if in addition $\Wcostname$ fulfills the property from Definition  \ref{ass:7.1} below. 
 \begin{definition}
 \label{ass:7.1}
 We say that 
 the action space   $(\Xsp,\Wcostname)$ 
 has the approximate mid-point property if 
\begin{equation}
\label{middle-point-condition}
\forall\, \rho>0 \ \forall\, u,v \in \Xsp \ \forall\, 0 < \epsilon \ll 1  \ \exists\, w \in \Xsp\, : \ \ \Wcost{\tfrac \rho2}{u}{w} + \Wcost{\tfrac \rho2}{w}{v} \leq \Wcost \rho uv + \epsilon\,.
\end{equation}
\end{definition}
Condition \eqref{middle-point-condition} 
mimicks the usual approximate mid-point property
for metric spaces: 
for any $\rho\in (0,+\infty)$ and every couple of points $u,\,v\in \Xsp$ and 
for any assigned threshold $\epsilon$, we find an `intermediate' point between $u$ and $v$ such that the sum of the costs for connecting $u$ to $w$ and $w$ to $v$ over intervals of half-length $\tfrac \rho2$
does not exceed 
of $\epsilon$
the cost for connecting $u$ to $v$ over the whole interval of length $\rho$. 
\par
\begin{theorem}
\label{thm:dynamic}
Let us suppose that 
the action space $(\Xsp,\Wcostname)$
is complete and has the approximate mid-point property. 
Then, 
for all $\tau,\eta>0$ and for all $u_0,u_1 \in \Xsp$
\begin{equation}
\label{dyn-prop}
\begin{gathered}
\text{there exists } \omega: [0,\tau] \to \Xsp \text{ with } \omega (0) = u_0, \ \omega(1) = u_1 
\text{ such that } 
\\
\Wcost \tau {u_0}{u_1}  \leq  \VarW \omega 0\tau \leq \Wcost \tau {u_0}{u_1} + \eta\,.
\end{gathered}
\end{equation}
In particular 
\begin{equation}
\label{dyn-BV}
\Wcost \tau {u_0}{u_1}  = \inf\Big\{  \VarW \Theta 0\tau \, : \ \Theta: [0,\tau] \to \Xsp, \ \Theta(0) =u_0, \ \Theta(\tau) =u_1\Big\}\,. 
\end{equation}
\end{theorem}
\begin{proof}
Without loss of generality, we may assume that $\tau=1$. Let us fix a threshold $\eta>0$: to construct the curve $\omega$ we will resort to  diadic partitions
\[
\diad n = \{ 0,\frac1{2^n}, \ldots, \frac j{2^n}, \ldots 1\}, \qquad n \in \N
\]
of the interval $[0,1]$. Indeed, we start by defining `discrete' curves defined on $\diad n$, in the following way: we pick a sequence 
 $(\overline{\eta}_n)_n$  such that 
\[
\sum_{n=0}^{\infty} 2^n\, \overline{\eta}_n =\eta 
\]
and perform the following steps:
\begin{itemize}
\item[\textbf{Step $0$:}] 
We apply \eqref{middle-point-condition} with $\rho =1  = \tfrac1{2^0}$, $u=u_0$, $v=u_1$,
and $\epsilon: = \overline{\eta}_0$,  thus finding a point $w \doteq w_{1/2}$ such that 
\[
\Wcost{\tfrac 12}{u_0}{w_{1/2}} + \Wcost{\tfrac 12}{w_{1/2}}{u_1} \leq \Wcost 1 {u_0}{u_1} + \overline{\eta}_0\,.
\]
Then, we define $\omega_0:  \diad 1 = \{ 0,\frac12, 1\} \to \Xsp$ by 
\[
\omega_0(0) : = u_0, \quad  \omega_0(\tfrac 12 ) : = w_{1/2}, \quad \omega_0(1) : = u_1\,.
\]
Clearly, we have 
\begin{equation}
\label{step0}
\Wcost{\tfrac 12}{\omega_0(0)}{\omega_0(\tfrac12)} + \Wcost{\tfrac 12}{\omega_0(\tfrac12)}{\omega_0(1)} \leq \Wcost 1 {u_0}{u_1} +  \overline{\eta}_0\,.
\end{equation}
\item[\textbf{Step $1$:}] We apply \eqref{middle-point-condition} with $\rho =\tfrac12 $, $u=\omega_0(0) =u_0$, $v=\omega_0(\tfrac12)$,
and $\epsilon: =  \overline{\eta}_1 $, thus obtaining a point $w \doteq w_{1/4}$ such that 
\[
\Wcost{\tfrac 1{2^2}}{\omega_0(0)}{w_{1/4}} + \Wcost{\tfrac 1{2^2}}{w_{1/4}}{\omega_0(\tfrac12)} \leq \Wcost {\tfrac12} {\omega_0(0)}{\omega_0(\tfrac12)} +  \overline{\eta}_1\,.
\]
We also apply 
\eqref{middle-point-condition} with $\rho =\tfrac12 $, $u=\omega_0(\tfrac12) $, $v=\omega_0(1) =u_1$,
and $\epsilon: =  \overline{\eta}_1 $, thus obtaining a point $w \doteq w_{3/4}$ such that 
\[
\Wcost{\tfrac 1{2^2}}{\omega_0(\tfrac12)}{w_{3/4}} + \Wcost{\tfrac 1{2^2}}{w_{3/4}}{\omega_0(1)} \leq \Wcost {\tfrac12} {\omega_0(\tfrac12)}{\omega_0(1)} +  \overline{\eta}_1\,.
\]
We then define 
$\omega_1:  \diad 2 = \{ 0, \frac14, \frac12, \frac34, 1\} \to \Xsp$ by 
\[
\omega_1(0) : = u_0, \quad  \omega_1(\tfrac 14 ) : = w_{1/4}, \quad \omega_1(\tfrac12) = \omega_0(\tfrac12) =w_{1/2}, \quad 
  \omega_1(\tfrac 34 ) : = w_{134},
 \quad \omega_1(1) : = u_1\,.
\]
By construction, we have 
\begin{equation}
\label{step1}
\begin{aligned}
& 
\Wcost{\tfrac 1{2^2}}{\omega_1(0)}{\omega_1(\tfrac14)} + \Wcost{\tfrac 1{2^2}}{\omega_1(\tfrac14)}{\omega_1(\tfrac12)}
+ \Wcost{\tfrac 1{2^2}}{\omega_1(\tfrac12 )}{\omega_1(\tfrac34)} +  \Wcost{\tfrac 1{2^2}}{\omega_1(\tfrac34 )}{\omega_1(1)}
\\
& 
\leq \Wcost {\tfrac12} {\omega_0(0)}{\omega_0(\tfrac12)} +  \Wcost {\tfrac12} {\omega_0(\tfrac12)}{\omega_0(1)}   + 2 \overline{\eta}_1
\\
& 
 \leq \Wcost 1 {u_0}{u_1} +  \overline{\eta}_0+2\overline{\eta}_1\,.
 \end{aligned}
\end{equation}
\item[\textbf{Step $n$:}] Let $t_{j-1} = \frac{j-1}{2^{n}}$ and $t_j  = \frac j{2^{n}}$, for $j\in \{1,\ldots, 2^{n}\}$, be  two nodes of the partition $\diad {n}$.
Clearly,
$t_{j-1} = \frac{m-2}{2^{n+1}} \in \diad {n+1}$ and $t_j = \frac{m}{2^{n+1}} \in \diad {n+1}$ with $m=2j$. Applying  \eqref{middle-point-condition} with $\rho =\tfrac1{2^n}$ and  
$\epsilon: = \overline{\eta}_n $, we find $w \doteq w_{(m-1)/2^{n+1}}$ such that 
\[
\begin{aligned}
&
\Wcost{\tfrac 1{2^{n+1}}}{\omega_{n-1}(\tfrac{m-2}{2^{n+1}})}{w_{(m-1)/2^{n+1}}}
 + \Wcost{\tfrac 1{2^{n+1}}}{w_{(m-1)/2^{n+1}}}{\omega_{n-1}(\tfrac m{2^{n+1}})}  
 \\
 & \leq \Wcost {\tfrac1{2^{n}}}{\omega_{n-1}(\tfrac{m-2}{2^{n+1}})}{\omega_{n-1}(\tfrac{m}{2^{n+1}})}
+
 \overline{\eta}_n
 \\
& 
= \Wcost {\tfrac1{2^{n}}}{\omega_{n-1}(\tfrac{j-1}{2^{n}})}{\omega_{n-1}(\tfrac{j}{2^{n}})}
+
 \overline{\eta}_n\,.
\end{aligned}
\]
Repeating this construction for every pair $(t_{j-1},t_j)$ of consecutive nodes of the partition $\diad {n+1}$, we 
define 
$\omega_n:  \diad {n+1} = \{ 0, \ldots, \frac k{2^{n+1}}, \ldots 1\} \to \Xsp$ by 
\[
\begin{cases}
\omega_n(0): = u_0, & 
\\
\omega_n \left( \frac k{2^{n+1}} \right)  := \omega_{n-1} \left( \frac j{2^{n}} \right)  & \text{if } k \text{ is even with }  k  = 2j,
\\
\omega_n \left( \frac k{2^{n+1}} \right)  : = w_{k} & \text{if } k \text{ is odd}, 
\\
\omega_n(1): = u_1
\end{cases}
\]
The function $\omega_n$ satisfies
\begin{equation}
\label{stepn}
\begin{aligned}
\sum_{k=1}^{2^{n+1}} \Wcost{\tfrac 1{2^{n+1}}}{\omega_n(\tfrac{k-1}{2^{n+1}})}{\omega_n(\tfrac{k}{2^{n+1}})} 
& \leq \sum_{j=1}^{2^n}  \Wcost {\tfrac1{2^{n}}}{\omega_{n-1}(\tfrac{j-1}{2^{n}})}{\omega_{n-1}(\tfrac{j}{2^{n}})}
+
2^n\overline{\eta}_n\,.
\\
& 
\leq  \Wcost 1{u_0}{u_1} + \sum_{j=0}^n 2^j \overline{\eta}_j
\end{aligned}
\end{equation}
\end{itemize}
Let now $\diad{\did}: = \bigcup_{n=0}^\infty \diad{n}$, and let us define $\omega_{\did}: \diad{\did} \to \Xsp$ by
\[
\omega_{\did}(t_\ell) : = \omega_n(t_\ell) \quad \text{if } t_\ell \in \diad n\,.
 \]
 It follows from our construction that $\omega_{\did}$ is well defined; from 
\eqref{stepn}, and the fact the map $\Wcost{\cdot}{u}{v}$
is non-increasing for every $u,\, v \in \Xsp$, 
we gather that 
\begin{equation}
\label{step-infty}
\sum_{\ell=0}^\infty \Wcost {\tau}{\omega_{\did}(t_\ell)}{\omega_{\did}(t_{\ell+1})} \leq \Wcost 1{u_0}{u_1} + \sum_{j=0}^\infty 2^j \overline{\eta}_j \qquad \text{for all } \tau>0
\end{equation}
\par
We are now in a position to construct the curve $\omega$ by extending $\omega_{\did}$ to $[0,1]{\setminus}\diad {\did}$
in the following way: for every $t\in [0,1]{\setminus}\diad {\did}$ we pick
the sequence $(t_{\ell_h})_h \subset \diad {\did}$ with $t_{\ell_h} \to t$ as $h\to\infty$. It follows from \eqref{step-infty}
that 
\[
\begin{aligned}
\forall\, \tau>0 \ \ \forall\, \eps>0 \ \ 
\exists\, \bar{h} \in \N \\  \forall\, k,h\geq \bar h\, : \ \ 
&  \Wcost {\tau}{\omega_{\did}(t_h)}{\omega_{\did}(t_{k})} 
\\ &  \leq 
 \sum_{\ell=h}^{k-1} \Wcost {\tau}{\omega_{\did}(t_\ell)}{\omega_{\did}(t_{\ell+1})} 
 \leq \eps\,,
 \end{aligned}
\]
namely the sequence $(\omega_{\did}(t_{\ell_h}))_h$ is Cauchy in the  $\Wtop$-uniform structure. Since $(\Xsp, \Wcostname) $ is complete, 
$(\omega_{\did}(t_{\ell_h}))_h$ admits a limit w.r.t.\ the $\Wtop$-topology. Let
\[
\omega_\infty(t):=  \Wcostname-\lim_{h\to\infty}\omega_{\did}(t_{\ell_h})\,.
\]
Therefore, we define $\omega: [0,1]\to \Xsp$ via
\begin{equation}
   \label{def-omega}
   \omega(t) : = \begin{cases}
   \omega_{\did}(t) & \text{if } t \in \diad {\did},
\\
\omega_\infty(t) & \text{if } t \in [0,1]{\setminus}\diad {\did}.
   \end{cases}
\end{equation}
\par
Now, 
let $(s_m)_{m=1}^M  \in \mathscr{P}_f([0,1])$ be an arbitrary  partition of $[0,1]$: we have that 
\[
\exists\, \bar{n} \in \N \ \  \forall\, m \in \{1,\ldots, M\} 
\ \  \forall\, n \geq n  \ \ \exists\, t_n^m \in \diad n \, : \quad 
\begin{cases}
    |s_m {-}t_n^m| \leq \frac 1{2^n},
    \\
    s_m-s_{m-1} \geq \frac 1{2^{n+1}}\,.
\end{cases}
\]
Therefore,
\[
\begin{aligned}
\Wcost{s_{m}{-}s_{m-1}}{\omega(s_{m-1})}{\omega(s_m)}
 & = \lim_{n\to\infty} \Wcost{s_{m}{-}s_{m-1}}{\omega_{\did}(t_n^{m-1})}{\omega_{\did}(t_n^m)} 
 \\
 &
\leq \liminf_{n\to\infty} 
\Wcost{\tfrac{1}{2^{n+1}}}{\omega_{n}(t_n^{m-1})}{\omega_{n}(t_n^m)}\,,
\end{aligned}
\]
where we have used that $\Wcost{\cdot}uv$ is non-increasing for all $u,v \in \Xsp$ and the fact that $\omega_{n}(t_n^{l})= \omega_{\did}(t_n^{l}) \toW \omega(s_l)$ for $l \in \{ m-1, m\}$. All in all, we have 
\[
\begin{aligned}
& 
\sum_{m=1}^M \Wcost{s_{m}{-}s_{m-1}}{\omega(s_{m-1})}{\omega(s_m)}
\\ & 
\leq \liminf_{n\to\infty} \sum_{m=1}^M \Wcost{\tfrac{1}{2^{n+1}}}{\omega_{n}(t_n^{m-1})}{\omega_{n}(t_n^m)}
\\ & 
\leq \Wcost 1{u_0}{u_1} + \sum_{j=0}^\infty 2^j \overline{\eta}_j = \Wcost 1{u_0}{u_1} + \eta\,
\end{aligned}
\]
where the second estimate follows from \eqref{step-infty}. 
By the arbitrariness of $(s_m)_{m=1}^M  \in \mathscr{P}_f([0,1])$, we ultimately conclude that 
$
\VarW {\omega} 01 \leq \Wcost 1{u_0}{u_1} + \eta\,.$
In turn, 
since $\omega(0)=u_0$ and $\omega(1) = u_1$, we obviously have 
$\Wcost 1{u_0}{u_1} \leq \VarW {\omega} 01$, and thus \eqref{dyn-prop} follows. This finishes the proof.

\end{proof}
\subsection{Existence of curves with minimal action}
Suppose now that $\Xsp$ is compact with respect to a topology $\sigma$ w.r.t.\ which $\Wcostname$ is  lower semicontinuous (cf.\ \eqref{lower-semicont}). Then, we can exploit Theorem \ref{thm:dynamic} to construct
for all $\tau >0$ and $u_0,\, u_1 \in \Xsp$
an optimal  curve $\omega_{\opt}$ for the minimum problem \eqref{dyn-BV}. 
\begin{theorem}
\label{thm:optimality-BV}
    Let $\Wcostname$ be, in addition, uniformly superlinear in the sense of  Definition \ref{ass:ETI} and suppose that $\Xsp$ is endowed with a topology $\sigma$ such that $(\Xsp,\sigma) $ is compact and $\Wcostname$ is  $\sigma$-lower semicontinuous \eqref{lower-semicont}. 
    \par
    Then, for every $\tau>0$ and $ u_0,\, u_1 \in \Xsp$ 
    there exists $\omega_{\opt}: [0,\tau] \to \Xsp$ such that 
    \begin{equation}
    \label{optimality}
    \Wcost \tau {u_0}{u_1} =\VarW {\omega_\opt} 0\tau = \min\{  \VarW \Theta 0\tau \, : \, \Theta: [0,\tau] \to \Xsp, \ \Theta(0) =u_0, \ \Theta(\tau) =u_1\}\,.
    \end{equation}
\end{theorem}
\begin{proof}
    Applying Thm.\ \ref{thm:dynamic} with $\eta = \eta_n = \frac1n$, $n\geq1$, we construct a sequence $(\omega^n)_n$ of finite-action curves such that 
    \[
\lim_{n\to\infty}\VarW {\omega^n} 0\tau =  \Wcost \tau {u_0}{u_1}\,.
    \]
  Due to \eqref{estimate-variations}, we have that 
  \[
\sup_n   \mathrm{Var}_{\Wmetric}(\omega^n;[0,\tau])<+\infty\,.
  \]
  Therefore, the sequence $(\omega^n)_n$ satisfy the conditions of 
  \cite[Thm.\ 3.2]{MaiMie05EREM}, and we infer that there exist
  a subsequence $(\omega^{n_k})_k$ and 
  $\omega^\infty : [0,\tau] \to \Xsp$ such that 
  \begin{equation}
  \label{sigma-convergence}
       \omega^{n_k}(t) \tosi\omega^\infty(t) \qquad \text{for every $t\in [0,T]$}
  \end{equation} 
  (where $\tosi$ indicates convergence in the $\sigma$-topology).
  We now show that 
  \begin{equation}
  \label{lsc-Var}
      \VarW {\omega^\infty} 0\tau \leq \liminf_{n\to\infty}\VarW {\omega^n} 0\tau\,.
  \end{equation}
  For this, we proceed as in the proof of Theorem \ref{thm:dynamic} and 
  fix any partition $(s_m)_{m=1}^M  \in \mathscr{P}_f([0,\tau])$. Due to \eqref{sigma-convergence} and the $\sigma$-lower semicontinuity \eqref{lower-semicont} of  $\Wcostname$, we find that for every $m=1,\ldots, M$
  \[
\Wcost{s_m{-}s_{m-1}}{\omega^\infty(s_{m-1})}{\omega^\infty(s_{m})} \leq
\liminf_{k\to\infty} \Wcost{s_m{-}s_{m-1}}{\omega^{n_k}(s_{m-1})}{\omega^{n_k}(s_{m})}\,.
  \]
  Then, \eqref{lsc-Var} follows from  adding the above estimate over 
  $m=1,\ldots, M$ and using the arbitrariness of  $(s_m)_{m=1}^M  \in \mathscr{P}_f([0,\tau])$. 
  We have thus shown that 
  \[
   \VarW {\omega^\infty}0\tau \leq  \Wcost \tau {u_0}{u_1}\,.
  \]
  On the other hand, since by \eqref{sigma-convergence}
 we have $\omega^\infty(0)=u_0$ and $\omega^\infty(\tau)=u_1$, the converse of the above inequality hold, so that 
 $  \VarW {\omega^\infty}0\tau =  \Wcost \tau {u_0}{u_1}$ and we may set $\omega_\opt: = \omega^\infty$. 
  \end{proof}

Ultimately, under the additional property from Definition \ref{ass:ETI} we obtain the existence of geodesics. 
\begin{corollary}
\label{cor:7.5}
    In addition to the assumptions of Thm.\ \ref{thm:optimality-BV}, suppose that 
    $\Wcostname$ is uniformly superlinear in the sense of  Definition \ref{ass:ETI}. Then,
    for every $\tau>0$ and $ u_0,\, u_1 \in \Xsp$ 
    there exists $\omega_\opt \in \AC_{\Wcostname}([0,\tau];\Xsp)$ such that 
    \begin{equation}
    \label{optimality-AC}
    \begin{aligned}
    \Wcost \tau {u_0}{u_1}  & = \int_0^\tau  \smd {\omega_\opt} t \dd t 
    \\
    & = \min \left\{ \int_0^\tau  \smd {\Theta} t \dd t  \, : \, \Theta \in \AC_{\Wcostname}([0,\tau]; \Xsp), \ \Theta(0) =u_0, \ \Theta(\tau) =u_1\right\}\,.
        \end{aligned}
    \end{equation}
\end{corollary}
\begin{proof}
It suffices to combine Theorem \ref{thm:optimality-BV} with Theorem \ref{thm:AC1}
 and Theorem \ref{l:consistency}. 
 \end{proof}

{\small

\markboth{References}{References}

\bibliographystyle{alpha}

\bibliography{ricky_lit}
}

\end{document}